\documentclass[12pt,a4paper,fleqn]{article}
\usepackage{amsmath}
\usepackage{amsthm}
\usepackage{amssymb}
\usepackage{amsfonts}
\usepackage{enumerate}
\usepackage{graphicx}
\usepackage{color}
\usepackage{rotating}
\usepackage{stmaryrd}
\newtheorem{theorem}{Theorem}[section]
\newtheorem*{theorem*}{Theorem}
\newtheorem{lemma}[theorem]{Lemma}
\newtheorem*{lemma*}{Lemma}
\newtheorem{corollary}[theorem]{Corollary}
\newtheorem*{corollary*}{Corollary}
\newtheorem{proposition}[theorem]{Proposition}
\newtheorem*{proposition*}{Proposition}

{\theoremstyle{definition}

\newtheorem*{definition*}{Definition}

\newtheorem*{example*}{Example}
\newtheorem*{examples*}{Examples}
\newtheorem{remark}[theorem]{Remark}
\newtheorem*{remark*}{Remark}
\newtheorem{result}{Result}

\newtheorem*{question*}{Question}

\newtheorem*{problem*}{Problem}

\newtheorem*{notation*}{Notation}

\newtheorem*{note*}{Note}

\newtheorem*{algorithm*}{Algorithm}
}
\renewcommand{\arraystretch}{0.75}

\newcommand{\Colon}{\:\mathbin{:}\:}
\newcommand{\Sum}{\displaystyle\sum}

\newcommand{\shape}{\mathbin{\mbox{\rm sh}}\,}
\newcommand{\ic}{\mathbin{{\rm ic}}}
\newcommand{\oc}{\mathbin{{\rm oc}}}

\newcommand{\LEQ}{\leqslant}

\let\unlhd\trianglelefteqslant
\title{On embedding certain Kazhdan--Lusztig cells of $S_n$ into cells of $S_{n+1}$}
\author{T.P.McDonough%
\thanks{%
\textit{Department of Mathematics,
Aberystwyth University,
Aberystwyth SY23 3BZ,
United Kingdom.}
{E-mail: tpd@aber.ac.uk}
}
\ and
C.A.Pallikaros%
\thanks{%
\textit{Department of Mathematics and Statistics,
University of Cyprus,
P.O.Box 20537,
1678 Nicosia,
Cyprus.}
{E-mail: pallikar@ucy.ac.cy}
}
}
\date{December 18, 2017}
\begin{document}
\maketitle
\begin{abstract}
In this paper, we consider a particular class of Kazhdan-Lusztig cells
in the symmetric group $S_n$, the cells containing involutions
associated with compositions $\lambda$ of $n$.
For certain families of compositions we are able to give an explicit
description of the corresponding cells by obtaining reduced forms for
all their elements.
This is achieved by first finding a particular
class of diagrams $\mathcal{E}^{(\lambda)}$ which lead to a subset of
the cell from which the remaining elements of the
cell are easily obtained.
Moreover, we show that for certain cases of
related compositions $\lambda$ and $\hat{\lambda}$ of $n$ and
$n+1$ respectively, the members of $\mathcal{E}^{(\lambda)}$ and
$\mathcal{E}^{(\hat{\lambda})}$ are also related in an analogous
way.
This allows us to associate certain cells in $S_n$ with cells in
$S_{n+1}$ in a well-defined way, which is connected to the induction and restriction of cells.
\\[1.5ex]
\end{abstract}
\begin{section}
{Introduction}
\label{sec:1c}
In \cite{KLu79}, when investigating the representations of a Coxeter
group and its associated Hecke algebra, Kazhdan and Lusztig introduced the \emph{left cells}, \emph{right cells} and
\emph{two-sided cells} of a Coxeter group.
The cell to which an element of the symmetric group $S_n$ belongs can be determined by
examining the tableaux resulting from an application of the
\emph{Robinson-Schensted process} to that element.
Also, the elements of a cell can be computed by applying the reverse
of the Robinson-Schensted process to a suitable selection of
tableaux pairs.
This, however, does not provide a straightforward way of obtaining reduced expressions for the elements of these cells.
\par
In this paper, we provide an alternative process for determining the
elements of a selection of cells.
These are cells which are associated in a
natural way with compositions of $n$, as the unique involution they contain is an element of longest length in the corresponding standard parabolic subgroup of $S_n$.
We show how to construct a certain set $\mathcal{E}^{(\lambda)}$ of
diagrams associated with the composition $\lambda$.
A diagram $D$, which is a subset of $\mathbb Z^2$, is a generalization of a Young diagram, in particular, we do not insist that $\mu_D''=\lambda_D'$ (where $\lambda_D$ and $\mu_D$ are respectively the row- and column-compositions of $D$).
Each diagram $D$ in $\mathcal{E}^{(\lambda)}$, where $\lambda=\lambda_D$, uniquely determines an
element $w_D$ of $S_n$ and all the elements $w_{J(\lambda)}w_D$ belong to the same cell.
We say the elements $w_D$, with $D\in\mathcal E^{(\lambda)}$, form the \emph{rim} of the
cell.
The remaining elements of the cell are the products of $w_{J(\lambda)}$ with the prefixes of the elements
in the rim.
\par
In this way we are able to give reduced expressions for all the
elements in cells corresponding to certain families of compositions
$\lambda$ (examples of these are given at the end of Section~3).
The techniques introduced earlier on in Section~3 (these are based on the work in Schensted~\cite{Schensted1961} and Greene~\cite{Greene1974} on increasing and decreasing subsequences) are crucial in achieving this and in fact they provide an extension to some of the techniques used in~\cite{MPa15}.
The main contribution of the paper is to show, by the use of the above techniques, that for certain related compositions $\lambda$ and $\hat{\lambda}$ of $n$ and
$n+1$ respectively, the members of $\mathcal{E}^{(\lambda)}$ and
$\mathcal{E}^{(\hat{\lambda})}$ are also related in an analogous
way.
This allows us to associate certain cells in $S_n$ with cells in
$S_{n+1}$ in a well-defined way which is in fact connected with the induction and restriction of Kazhdan-Lusztig cells (see Barbasch and Vogan~\cite{BVo83}).

Before discussing briefly the main results of the paper (Theorems~\ref{TheoremReplaceThm6.1} and~\ref{thm:6.5c} in Section~4) we need to introduce some more notation.
For $\lambda=(\lambda_1,\ldots,\lambda_r)$ a composition of $n$, let $W_{J(\lambda)}$ be the standard parabolic (Young) subgroup of $S_n$ corresponding to $\lambda$ and let $w_{J(\lambda)}$ be the longest element of $W_{J(\lambda)}$.
Also let $\mathfrak X_{J(\lambda)}$ be a complete set of distinguished right coset representatives of $W_{J(\lambda)}$ in $S_n$.
The right cell, $\mathfrak C(\lambda)$, containing $w_{J(\lambda)}$ has has form $w_{J(\lambda)}Z(\lambda)$ for some subset $Z(\lambda)$ of $\mathfrak X_{J(\lambda)}$.
Also set $Y(\lambda)=\{x\in Z(\lambda):$ $x$ is not a prefix of any other $y\in Z(\lambda)\}$.
Then $Y(\lambda)=\{w_D:\ D\in\mathcal{E}^{(\lambda)}\}$ with $\mathcal E^{(\lambda)}$ and $w_D$ as above.
(Note that for a diagram $D$, $w_D$ is the unique element of minimum length in the double coset $W_{J(\lambda_D)}w_DW_{J(\mu_D)}$, and this double coset has the trivial intersection property.)

We will associate to the composition $\lambda$ of $n$ the compositions $\lambda_*$ and $\lambda^{(k)}$ of $n+1$, where $\lambda_*=(\lambda_1,\ldots,\lambda_r,1)$ and $\lambda^{(k)}=(\lambda_1,\ldots,\lambda_{k-1},\lambda_k+1,\lambda_{k+1},\ldots,\lambda_r)$ for  $1\le k\le r$.
Moreover, given a diagram $D$ with row-composition $\lambda_D=\lambda$ we form diagrams $D'$ (with $\lambda_{D'}=\lambda_*$) and $D^{(k)}$ (with $\lambda_{D^{(k)}}=\lambda^{(k)}$) by inserting a new node to $D$ in a well-defined way.


In Theorem~\ref{TheoremReplaceThm6.1} we define, via diagram pairs $(D,D')$, an injection $\theta_*$ from $Z(\lambda)$ to $Z(\lambda_*)$ (this induces an injection from $\mathcal{E}^{(\lambda)}$ to $\mathcal{E}^{(\lambda_*)}$) which satisfies $Y(\lambda)\theta_*\subseteq Y(\lambda_*)\subseteq Z(\lambda)\theta_*$.
So, $\theta_*$ not only relates the rim of the cells $\mathfrak C(\lambda)$ and $\mathfrak C(\lambda_*)$ but also gives a way of determining (and obtaining reduced forms for) all the elements of cell $\mathfrak C(\lambda_*)$ once we have the corresponding information about cell $\mathfrak C(\lambda)$.
In addition, we give the connection of the process described in this theorem with the induction of cells.

Finally, in Theorem~\ref{thm:6.5c} (under the assumption that $\lambda=(\lambda_1,\ldots,\lambda_r)$ is a composition of $n$ with $\lambda_k\ge\lambda_i$ for $1\le i\le r$) we define, now via diagram pairs $(D,D^{(k)})$ an injection $\theta$ from $Z(\lambda)$ to $Z(\lambda_*)$ and give some sufficient conditions for $\theta$ to satisfy $Y(\lambda)\theta \subseteq Y(\lambda^{(k)})$.
In particular, when $k=1$, we have equality $Y(\lambda)\theta=Y(\lambda_*)$
(thus, in this case, $\theta$ induces a bijection from $\mathcal{E}^{(\lambda)}$ to $\mathcal{E}^{(\lambda^{(k)})}$) and this allows us to determine the rim of $\mathfrak C(\lambda^{(k)})$, and hence obtain reduced forms for all the elements of this cell, just from knowledge of the rim of $\mathfrak C(\lambda)$.
We also show the precise connection of this process with the restriction of cells.

\end{section}
\begin{section}
{Induction and restriction of Kazhdan-Lusztig cells of $S_n$}
\label{sec:2c}
We begin this section by recalling some of the basic concepts and
results of the Kazhdan-Lusztig representation theory of Coxeter groups
and Hecke algebras.
For any Coxeter system $(W,S)$, Kazhdan and Lusztig \cite{KLu79}
introduced three preorders $\LEQ_L$, $\LEQ_R$ and $\LEQ_{LR}$, with
corresponding equivalence relations $\sim_L$, $\sim_R$ and $\sim_{LR}$,
whose equivalence classes are called
\emph{left cells}, \emph{right cells} and \emph{two-sided cells},
respectively.

For basic concepts relating to Coxeter groups and Hecke algebras,
see Geck and Pfeiffer \cite{GPf00} and Humphreys \cite{Hum90}.
In particular, for a Coxeter system $(W,S)$,
$W_J=\langle J\rangle$ denotes the parabolic subgroup determined by a
subset $J$ of $S$, $w_J$ denotes the longest element of $W_J$,
$\mathfrak{X}_J$ denotes the set of minimum length elements in the
right cosets of $W_J$ in $W$
(the distinguished right coset representatives),
$\LEQ$ denotes the strong Bruhat order on $W$,
and $w<w'$ means $w\LEQ w'$ and $w\neq w'$ if $w,w'\in W$.
The pair $(W_J,J)$ is a Coxeter system whose length function is the same as the
restriction of the length function of $(W,S)$ to it;
consequently, $w_J$ is determined entirely by $J$.
Also recall the \emph{prefix relation} on the elements of~$W$:
if $x,y\in W$ we say that $x$ is a \emph{prefix} of $y$ if $y$ has a \emph{reduced form} beginning with a reduced form for $x$.

The Hecke algebra $\mathcal{H}$ corresponding to $(W,S)$ and
defined over the ring $A=\mathbf{Z}[q^{\frac{1}{2}},q^{-\frac{1}{2}}]$,
where $q$ is an indeterminate, has a free $A$-basis
$\{T_w\Colon w\in W\}$ and multiplication defined by the rules
\begin{equation*}
\label{eqn:1c}
\begin{tabular}{rl}
(i) & $T_wT_{w'}=T_{ww'}$ if $l(ww')=l(w)+l(w')$ and \\[2pt]
(ii) & $(T_s+1)(T_s-q)=0$ if $s\in S$.
\end{tabular}
\end{equation*}
The basis $\{T_w\Colon w\in W\}$ is called the \emph{$T$-basis} of
$\mathcal{H}$.
(See \cite{KLu79}).
\begin{result}
[{\cite[Theorem~1.1]{KLu79}}]
\label{res:1c}
$\mathcal{H}$ has a basis $\{C_w\Colon w\in W\}$, the \emph{$C$-basis},
whose terms have the form
$C_y =
\Sum_{x\LEQ y}
(-1)^{l(y)-l(x)}q^{\frac{1}{2}l(y)-l(x)}P_{x,y}(q^{-1})T_x$,
where $P_{x,y}(q)$ is a polynomial in $q$ with integer coefficients
of degree $\le \frac{1}{2}\left(l(y)-l(x)-1\right)$ if $x<y$ and
$P_{y,y}=1$.
\end{result}

The following result collects some useful propositions concerning cells.
For proofs of (i), (ii) and (iii), see \cite[2.3ac]{KLu79},
\cite[5.26.1]{Lus84b}, and \cite[Corollary~1.9(c)]{Lus87},
respectively.
See also \cite[Lemma~5.3]{Gec05} for a more elementary algebraic proof
of (iii) in the case that $W$ is the symmetric group.
\begin{result}[\cite{KLu79,Lus84b,Lus87}]
\label{res:2c}
\mbox{}\par
\begin{tabular}[b]{rp{5.25in}}
(i) &
If $x,y,z$ are elements of $W$ such that $x$ is a prefix of $y$,
$y$ is an prefix of $z$ and $x\sim_R z$ then $x\sim_R y$.
\\
(ii) &
If $J\subseteq S$, then the right cell containing $w_J$ is contained
in $w_J\mathfrak{X}_J$.
\\
(iii) &
If $W$ is a crystallographic group and $x,y\in W$ are such that
$x\sim_{LR}y$ and $x\LEQ_R y$ then $x\sim_R y$.
\end{tabular}
\end{result}
\par
Each cell of $W$ provides an integral representation of $W$, with the $C$-basis of $\mathcal H$ playing an important role in the construction of this representation; see
Kazhdan and Lusztig \cite[\S~1]{KLu79}. Barbasch and Vogan \cite{BVo83}
have addressed the question of induction and restriction of such
representations in relation to parabolic subgroups, where $W$ is a
Weyl group. They obtained the following results, which have natural
analogues for right cells.
We quote these results of Barbasch and Vogan as they are formulated in~\cite{Roi98}.

\begin{result}[{\cite[Proposition~3.15]{BVo83}}]
\label{res:4c}
Let $H$ be a parabolic subgroup of a Weyl group $W$.
Let $\mathfrak{C}$ be a Kazhdan-Lusztig left cell of $H$ and let
$\rho^{\mathfrak{C}}$ be its Kazhdan-Lusztig representation of $H$.
Let $L$ be the set of all representatives of minimal length of the
left cosets of $H$ in $W$.
Then $L\mathfrak{C}=\{rv\colon r\in L, v\in\mathfrak{C}\}$ is a union of
Kazhdan-Lusztig left cells.
The associated Kazhdan-Lusztig representation $\rho^{L\mathfrak{C}}$ is
isomorphic to the induced representation
$\rho^{\mathfrak{C}}\!\!\uparrow^{W}_{H}$.
\end{result}

\begin{result}[{\cite[Proposition~3.11]{BVo83}}]
\label{res:3c}
Let $H$ be a parabolic subgroup of a Weyl group $W$.
Let $R$ be the set of representatives of minimal length of the
right cosets of $H$ in $W$.
Let $\mathfrak{C}$ be a Kazhdan-Lusztig left cell of $W$, and let
$\rho^{\mathfrak{C}}$ be the Kazhdan-Lusztig representation of $W$
associated with $\mathfrak{C}$.
Then $\mathfrak{C}$ is a disjoint union of sets of the form
$\mathcal{D}r$ where $r\in R$ and $\mathcal{D}$ is a Kazhdan-Lusztig
cell of $H$.
The direct sum of the associated Kazhdan-Lusztig left representations
of $H$, $\bigoplus_{\mathcal{D}}\rho^{\mathcal{D}}$, corresponds to
the decomposition of the restricted representation
$\rho^{\mathfrak{C}}\!\!\downarrow_{H}^{W}$.
\end{result}

These results have been generalized to all Coxeter groups by
Roichman~\cite{Roi98} and Geck~\cite{Gec03}.

\medskip
We will focus our attention on the symmetric group.
For the basic definitions and background concerning partitions, compositions, Young diagrams and Young tableaux we refer to James~\cite{Jam78}, Fulton~\cite{Ful97} or Sagan~\cite{Sagan}.

All our partitions and compositions will be assumed to be \emph{proper}
(that is, with no zero parts).
We use the notation $\lambda\vDash n$ (respectively, $\lambda\vdash n$)
to say that $\lambda$ is a composition (respectively, partition) of
$n$.
For $\lambda\vDash n$ having $r'$ as its maximum part,
recall that the \emph{conjugate} composition
$\lambda'=(\lambda_1',\ldots,\lambda_{r'}')$ of $\lambda$ is defined by
$\lambda'_i=\left|\{j\Colon 1\leq j\leq r\mbox{ and }
i\leq\lambda_j\}\right|$ for $1\leq i\leq r'$.
It is immediate that $\lambda'$ is a partition of $n$ with $r'$ parts.
\par
If $\lambda$ and $\mu$ are partitions of $n$, write
$\lambda\unlhd\mu$ if, for all $k$,
$\sum_{1\leq i\leq k}\lambda_i \leq \sum_{1\leq i\leq k}\mu_i$.
This is the dominance order of partitions (see \cite[p.8]{Jam78}).
If $\lambda\unlhd\mu$ and $\lambda\neq \mu$, we write $\lambda\lhd\mu$.

In the case of the symmetric group $S_n$, the Robinson-Schensted
correspondence gives a combinatorial method of identifying the
Kazhdan-Lusztig cells.
The Robinson-Schensted correspondence is a bijection of $S_n$
to the set of pairs of standard Young tableaux
$(\mathcal{P},\mathcal{Q})$ of the same shape and with $n$ entries,
where the shape of a tableau is the partition counting the number of
entries on each row.
See \cite{Ful97} or \cite{Sagan}
for a good description of this correspondence.
Denote this correspondence by
$w\mapsto(\mathcal{P}(w),\mathcal{Q}(w))$.
Then $\mathcal{Q}(w)=\mathcal{P}(w^{-1})$.
The \emph{shape} of $w$, denoted by $\shape{w}$, is defined to be
the common shape of the Young tableaux $\mathcal{P}(w)$ and
$\mathcal{Q}(w)$.
The tableaux $\mathcal{P}(w)$ and $\mathcal{Q}(w)$ are called
the \emph{insertion tableau} and the \emph{recording tableau},
respectively, for $w$.

The following result in~\cite{KLu79} characterises the cells in $S_n$.
%
\begin{result}
[\cite{KLu79}, see also {\cite[Theorem A]{Ari00} or \cite[Corollary 5.6]{Gec05}}]
\label{res:5c}
If $\mathcal{P}$ is a fixed standard Young tableau then
the set $\{w\in W\Colon \mathcal{P}(w)=\mathcal{P}\}$
is a left cell of $W$ and the set
$\{w\in W\Colon\mathcal{Q}(w)=\mathcal{P}\}$ is a right cell of $S_n$.
Conversely, every left cell and every right cell arises in this way.
Moreover, the two-sided cells are the subsets of $W$ of the form
$\{w\in W\Colon\,\shape\mathcal{P}(w)\mbox{ is a fixed partition.}\}$
\end{result}
\par
The \emph{shape} $\shape{\mathfrak{C}}$ of a cell $\mathfrak{C}$ is
$\shape{w}$ for any $w\in\mathfrak{C}$.
\par
For the rest of the paper, unless explicitly mentioned otherwise,
$W$ and $W'$  will be the symmetric groups $S_{n}$ on $\{1,\dots,n\}$
and $S_{n+1}$ on $\{1,\ldots,n+1\}$, respectively, with the natural
embedding.
Let $s_i=(i,i+1)$ for $1\le i\le n$,
let $S=\{s_1,\ldots,s_{n-1}\}$ and $S'=\{s_1,\ldots,s_n\}$, so that
$(W,S)$ and $(W',S')$ are Coxeter systems.
Let $\mathfrak{X}'_K$ be the set of distinguished right
coset representatives of $W_K'$ in $W'$, for any subset $
K\subseteq S'$.
Note that $W$ is embedded naturally
in $W'$ as $W'_S$.
We write $\mathfrak{X}'$ for $\mathfrak{X}_{S}'$.
Then $\mathfrak{X}'=\{x_i\colon 1\le i\le n+1\}$ where, for $1\le i\le n+1$,
$x_i=(i,i+1,\dots,n,n+1)=s_{n}\cdots s_{i}$
(the empty product is $1$, by convention).

We will describe an element $w$ of $S_n$ in different forms:
as a word in the generators $s_1$, \ldots , $s_{n-1}$,
as products of disjoint cycles on $1,\ldots,n$,
and in \emph{row-form}
$[w_1,\ldots,w_n]$ where $w_i=iw$ for $i=1,\dots,n$.
Also if $\lambda=(\lambda_1, \ldots , \lambda_r)$ is a \textit{composition}
of $n$ with $r$ \emph{parts},
we define the subset $J(\lambda)$ of $S$ to be
$S\backslash \{s_{\lambda_1},s_{\lambda_1+\lambda_2},\ldots,
s_{\lambda_1+\ldots+\lambda_{r-1}}\}$.
We make similar definitions for $S_{n+1}$ and compositions of $n+1$.

For a Young diagram $D$ corresponding to the partition
$\lambda=(\lambda_1,\dots,\lambda_r)$, let $\ic(D)$ and $\oc(D)$ be the
sets of \emph{inner corners} and \emph{outer corners}, respectively,
of $D$; that is,
\[
\begin{array}{rcl}
\ic(D) & = &
\{(i,\lambda_i))\colon 1\le i\le r-1
\mbox{ where } \lambda_{i}>\lambda_{i+1}\}
\cup
\{(r,\lambda_r)\},
\\
\oc(D) & = &
\{(1,\lambda_1+1)\}
\cup
\{(i,\lambda_i+1))\colon
2\le i\le r \mbox{ where } \lambda_{i-1}>\lambda_i\}
\cup
\{(r+1,1)\}.
\end{array}
\]
We denote by $<$ the total order on the nodes of a diagram given by
$(i,j)<(i',j')$ if, and only if, $i<i'$ or $i=i'$ and $j<j'$.
\begin{proposition}
\label{prop:2.1c}
Let $\mathfrak{C}$ be a right cell of $W$,
let $A$ be the recording tableau of elements of $\mathfrak{C}$ and
let $D$ be its underlying diagram.
For each $k\in \oc(D)$, let $A_k$ be the tableau obtained from $A$ by
adding the entry $n+1$ at node $k$ and let $\mathfrak{C}_k$ be the
right cell of $W'$ corresponding to the recording tableau $A_k$.
Then $\mathfrak{C}\mathfrak{X}'=\bigcup_{k\in\oc(D)}\mathfrak{C}_k$.
\par
Furthermore, if $k,k'\in\oc(D)$ and $k<k'$ then
$\shape{\mathfrak{C}_{k'}}\lhd\shape{\mathfrak{C}_{k}}$.
\end{proposition}
\begin{proof}
Let $w=[w_1,\dots,w_n]\in\mathfrak{C}$.
Then $\mathcal{Q}(w)=A$.
Let $\mathcal{P}(w)=B$.
If $1\le i\le n+1$, then $wx_i=[\bar w_1,\dots,\bar w_n,i]$ where
$\bar w_j=w_j$ if $1\le w_j<i$ and
$\bar w_j=w_j+1$ if $i\le w_j\le n$.
In determining the tableaux associated with $wx_i$, the tableaux which
arise after the first $n$ insertions are $\bar B=Bx_i$ and $A$.
Hence, $\mathcal{Q}(wx_i)=A_k$ for some $k\in \oc(D)$.
So, $wx_i\in\mathfrak{C}_k$.
\par
On the other hand, suppose that $w'\in\mathfrak{C}_k$ where $k\in\oc(D)$.
Write $w'=[w_1',\dots,w_{n+1}']$ and let $i=w_{n+1}'$.
Following the first $n$ insertions of $w'$ the recording tableau is
$A$.
Hence, $w=w'x_i^{-1}\in S_n$ and $\mathcal{Q}(w)=A$, so that
$w\in\mathfrak{C}$ and $w'\in\mathfrak{C}x_i$.
\par
It follows that
$\mathfrak{C}\mathfrak{X}'=\bigcup_{k\in \oc(D)}\mathfrak{C}_k$.
The final sentence in the proposition is immediate.
\end{proof}
\begin{proposition}
\label{prop:2.2c}
Let $\mathfrak{C}$ be a right cell of $W'$ and
let $A$ be the recording tableau of elements of $\mathfrak{C}$
and let $D$ be its underlying diagram.
For each $k\in \ic(D)$, if $i(k)$ is the entry on the first row of $A$
removed by reverse inserting from node $k$ and $A'$ is the resulting
tableau, let $d_k=x_{i(k)}^{-1}$ and let $A_k=A'd_k$
(so that $A_k$ is a standard Young tableau on $1,\dots,n$).
Let $\mathfrak{C}_k$ be the right cell of $W$ corresponding to the
recording tableau $A_k$.
Then $\mathfrak{C}=\bigcup_{k\in \ic(D)}d_k\mathfrak{C}_k$.
\par
Furthermore, if $k,k'\in\ic(D)$ and $k<k'$ then
$\shape{\mathfrak{C}_{k}}\lhd\shape{\mathfrak{C}_{k'}}$
and $d_{k}\LEQ d_{k'}$.
\end{proposition}
\begin{proof}
Let $w\in \mathfrak{C}$.
Then $\mathcal{Q}(w)=A$.
Let $\mathcal{P}(w)=B$.
Then $\mathcal{P}(w^{-1})=A$ and $\mathcal{Q}(w^{-1})=B$.
Suppose $w^{-1}=[v_1,\dots,v_{n+1}]$.
Let $k\in\ic(D)$ be the node at which $B$ has entry $n+1$.
Let $B'$ be obtained from $B$ by removing the entry $n+1$.
Also let $A'$ be obtained from $A$ by reverse-insertion from node $k$,
and let $i(k)$ be the entry removed from the first row of $A$ by this
process.
Then $v_{n+1}=i(k)$.
Moreover, $A'$ and $B'$ are the insertion and recording tableaux,
respectively, arising from the insertion of the first $n$ entries
of the row-form of~$w^{-1}$.
Hence, $w^{-1}d_k=w^{-1}x_{i(k)}^{-1}\in W$ and
$\mathcal{P}(w^{-1}d_k)=A'd_k=A_k$.
So, $d_k^{-1}w\in W$ and $\mathcal{Q}(d_k^{-1}w)=A_k$.
That is, $w\in d_k\mathfrak{C}_k$.
\par
The preceding argument is easily seen to be reversible.
Hence, $\mathfrak{C}=\bigcup_{k\in \ic(D)}d_k\mathfrak{C}_k$.
The first part of the final sentence in the proposition is immediate.
Since the reverse insertion path from node $k'$ must pass the row of
$k$ weakly left of that node, it must remain weakly left of the
reverse insertion path from node $k$. Hence, $i(k')\le i(k)$.
So, $d_{k}=s_{i(k)}\cdots s_{n-1}s_{n}\LEQ s_{i(k')}\cdots s_{n-1}s_{n}
=d_{k'}$.
\end{proof}
\end{section}
\begin{section}
{Paths and admissible diagrams: determining the rim of certain cells}
\label{sec:3c}

We recall the generalizations of the notions of diagram and tableau,
commonly used in the basic theory, which we described in \cite{MPa15}.
A \emph{diagram} $D$ is a finite subset of $\mathbb{Z}^2$.
We will assume, where possible, that $D$ has no empty rows or columns.
These are the principal diagrams of \cite{MPa15}.
We will also assume that both rows and columns of $D$ are indexed
consecutively from 1; a node in $D$ will be given coordinates $(a,b)$ where $a$ and $b$ are the indices respectively of the row and column which the node belongs to
(rows are indexed from top to bottom and columns from left to right).
For a principal diagram $D$ we denote by $c_D$ and $r_D$ the number of columns and rows of $D$ respectively.
The \emph{row-composition} $\lambda_D$ (respectively,
\emph{column-composition} $\mu_D$) of $D$ is defined by
setting $\lambda_{D,k}$ (respectively, $\mu_{D,k}$) to be the number of nodes on
the $k$-th row (respectively, column) of $D$ for $1\le k\le r_D$ (respectively, $1\le k\le c_D$).
If $\lambda$ and $\mu$ are compositions,
we will write $\mathcal{D}^{(\lambda,\mu)}$ for the set
of (principal) diagrams $D$ with $\lambda_D=\lambda$ and $\mu_D=\mu$.
We also define $\mathcal{D}^{(\lambda)}$ to be the set
$\bigcup_{\mu\vDash n}\mathcal{D}^{(\lambda,\mu)}$ of (principal) diagrams with $\lambda_D=\lambda$.
A \emph{special diagram} is a diagram obtained from a Young diagram
by permuting the rows and columns.
Special diagrams are characterised in the following proposition.
\begin{result}[{\cite[Proposition~3.1]{MPa15}},\ %
{\rm Compare~\cite[Lemma~5.2]{DMP10}}]
\label{res:6c}
Let $D$ be a diagram. The following statements are equivalent.
(i)~$D$ is special;
(ii)~$\lambda_D''=\mu_D'$;
(iii)~for every pair of nodes $(i,j),(i',j')$ of $D$ with $i\neq i'$
and $j\neq j'$, at least one of $(i',j)$ and $(i,j')$ is also a node
of $D$.
\end{result}
If $D$ is a diagram with $n$ nodes,
a \emph{$D$-tableau} is a bijection
$t\Colon D\rightarrow \{1,\ldots,n\}$ and
we refer to $(i,j)t$, where $(i,j)\in D$, as the
$(i,j)$-\emph{entry} of $t$.
The group $W$ acts on the set of $D$-tableaux in the obvious
way---if $w\in W$,
an entry $i$ is replaced by $iw$ and $tw$ denotes the tableau
resulting from the action of $w$ on the tableau $t$.
We denote by $t^{D}$ and  $t_{D}$ the two $D$-tableaux obtained by
filling the nodes of $D$ with $1,\ldots,n$ by rows and by columns,
respectively, and we write $w_{D}$ for the element of $W$ defined by
$t^{D}w_{D}=t_{D}$.

Now let $D$ be a diagram and let $t$ be a \emph{$D$-tableau}.
We say $t$ is \emph{row-standard} if it is increasing on rows.
Similarly, we say $t$ is \emph{column-standard} if it is increasing
on columns.
We say that $t$ is \emph{standard} if $(i',j')t\le (i'',j'')t$
for any $(i',j'),(i'',j'')\in D$ with $i'\leq i''$ and $j'\leq j''$.
Note that a standard $D$-tableau is row-standard and column-standard,
but the converse is not true, in general.

\begin{result}[{\cite[Proposition 3.5]{MPa15}.
Compare \cite[Lemma 1.5]{DJa86}}]
\label{res:7c}
Let $D$ be a diagram.
Then the mapping $u\mapsto t^{D}u$ is a bijection of the set of prefixes of $w_{D}$ to the set of standard
$D$-tableaux.
\end{result}
\par
In general, an element of $W$ will have an expression of the form
$w_D$ for many different diagrams $D$ of size $n$.
The following result shows how to locate suitable diagrams.
\par
\begin{result}[{\cite[Proposition~3.7]{MPa15}}]
\label{res:8c}
Let $\lambda\vDash n$ and let $d\in \mathfrak{X}_{J(\lambda)}$.
Then $d=w_D$ for some diagram $D\in\mathcal{D}^{(\lambda)}$.
\end{result}

The proof involves the construction of a principal diagram $D(d,\lambda)$ with
the desired properties. This is formed by partitioning the row-form
of $d$ in parts of sizes corresponding to $\lambda$, placing these
parts on consecutive rows and moving the entries on the rows minimally
to make a tableau of the form $t_D$.

Now let $d\in \mathfrak{X}_{J(\lambda)}$ and denote by $\mathcal{D}^{(\lambda)}_{d}$
\label{def:diags-b}
the set of principal diagrams $D\in\mathcal{D}^{(\lambda)}$ for
which $w_D=d$.
The following result, which will turn out to be useful later on, tells us in what way two elements of $\mathcal{D}^{(\lambda)}_{d}$ can differ from one another.
\begin{result}[{\cite[Proposition~3.8]{MPa15}}]
\label{res:ischia}
Let $\lambda\vDash n$, let $d\in\mathfrak{X}_{J(\lambda)}$, let
$D=D(d,\lambda)$ and let $E\in\mathcal{D}^{(\lambda)}_{d}$.
Then the set of columns of $E$ may be partitioned into sets of
consecutive columns so that, for $j\geq 1$,
\par
\begin{tabular}{rp{5.68in}}
(i) & for any two columns in the $j$-th set, the nodes in the column
with lesser column index have row indices which are less than all the
indices of the nodes in the column with greater column index;
\\[3pt]
(ii) & the row indices of the nodes occurring in columns of the
$j$-th set are precisely the row indices of the nodes in the $j$-th
column of $D$.
\end{tabular}
\\
In particular, $D$ is the unique diagram in
$\mathcal{D}^{(\lambda)}_{d}$ with the minimum number of columns.
\end{result}

We illustrate some of the above concepts with an example.

\begin{example*}
\label{ex:4.3c}
Let $\lambda=(3^2,2,1)$ and
$d$ $=$ 
$[3,4,7,2,6,8,1,9,5]\in\mathfrak X_{J(\lambda)}$.
From the `corresponding' row-standard tableau,
$
\begin{array}{lll}
3& 4& 7\\
2& 6& 8\\
1& 9\\
5
\end{array}
$,
we produce the tableau
$
\begin{array}{llllll}
  &  & 3& 4&  & 7\\
  & 2&  &  & 6& 8\\
 1&  &  &  &  & 9\\
  &  &  & 5
\end{array}
$
using the procedure in~\cite[Proposition~3.7]{MPa15} (see comments after Result~\ref{res:8c}).
Now let
\[
D=\begin{array}{llllll}
       &       & \times& \times&       & \times\\
       & \times&       &       & \times& \times\\
 \times&       &       &       &       & \times\\
       &       &       & \times
\end{array}
\quad\mbox{and}\quad
E=
\begin{array}{llllllll}
       &       & \times& \times&       &       & \times &\\
       & \times&       &       &       & \times&        & \times\\
 \times&       &       &       &       &       &        & \times\\
       &       &       &       & \times
\end{array}.
\]
Then $D=D(d,(3^2,2,1))$ and $w_D=w_{E}=d$ (compare with Result~\ref{res:ischia}).
Note that the 4-th and 6-th sets of columns in $E$ referred to in
Result~\ref{res:ischia} are $\{4,5\}$ and $\{7,8\}$.
We also have $D=\{(1,3),\, (1,4),\, (1,6),\, (2,2),\, (2,5),\, (2,6),\, (3,1),\, (3,6),\, (4,4)\}$, $c_D=6$ and $r_D=4$.

Now let $e_1=[2,3,4,1,6,7,5,8,9]$ and $e_2=[2,5,6,1,4,7,3,8,9]$.
Then,
\[
t^De_1=
\begin{array}{llllll}
  &  & 2& 3&  & 4\\
  & 1&  &  & 6& 7\\
 5&  &  &  &  & 8\\
  &  &  & 9
\end{array}
\quad\mbox{and}\quad
t^De_2=
\begin{array}{llllll}
  &  & 2& 5&  & 6\\
  & 1&  &  & 4& 7\\
 3&  &  &  &  & 8\\
  &  &  & 9
\end{array}.
\]
As $t^De_1$ is standard but $t^De_2$ is not standard, we conclude from Result~\ref{res:7c} that $e_1$ is a prefix of $d$ while $e_2$ is not a prefix of $d$.
\end{example*}

For a composition $\lambda$ of $n$, we define the following subsets
of $\mathfrak{X}_{J(\lambda)}$ and $\mathcal{D}^{(\lambda)}$:
\begin{equation*}
\renewcommand{\arraystretch}{1.0}
\label{eqn:2c}
\begin{array}{rcl}
Z(\lambda) & = &
\{e\in\mathfrak{X}_{J(\lambda)}\colon w_{J(\lambda)}e\sim_Rw_{J(\lambda)}\},
\\
Z_s(\lambda) & = &
\{e\in Z(\lambda)\colon e=w_D
\mbox{ for some special diagram } D\in\mathcal{D}^{(\lambda)}\},
\\
Y(\lambda) & = &
\{x\in Z(\lambda)\colon x \mbox{ is not a prefix of any other }
y\in Z(\lambda)\},
\\
Y_s(\lambda) & = &
Y(\lambda)\cap Z_s(\lambda),
\\
\mathcal{E}^{(\lambda)} & = &
\{D(y,\lambda)\colon y\in Y(\lambda)\}.
\end{array}
\end{equation*}
\par
In view of Result~\ref{res:2c}(i) and (ii), $Z(\lambda)$ is closed
under the taking of prefixes and $w_{J(\lambda)}Z(\lambda)$ is the
right cell of $W$ containing $w_{J(\lambda)}$.
We denote this right cell by $\mathfrak{C}(\lambda)$.
Note that $e\in Z(\lambda)$ if, and only if,
$e\in \mathfrak{X}_{J(\lambda)}$ and
${\cal Q}(w_{J(\lambda)}e)={\cal Q}(w_{J(\lambda)})$.
\par
A knowledge of $Y(\lambda)$ leads directly to $Z(\lambda)$
by determining all prefixes.
We call $Y(\lambda)$ the \emph{rim} of the cell
$\mathfrak{C}(\lambda)$.

\begin{remark}\label{ypartition}
In the case that $\lambda$ is a partition of $n$, it follows from
\cite[Lemma~3.3]{MPa05} that
$Y(\lambda)=Y_s(\lambda)=\{w_{D}\}$, where $D$ is the Young diagram with
$\lambda_D=\lambda$, and $Z(\lambda)$ is the set of prefixes of $w_{D}$.
\end{remark}
We now show how a knowledge of the increasing subsequences
of the row-form of $w_{J(\lambda_D)}e$, where $D$ is any diagram and
$e$ is a prefix of $w_D$, helps to determine whether
$w_{J(\lambda_D)}e$ is in the same right cell as $w_{J(\lambda)}$.
\begin{lemma}
\label{lem:3.1c}
Let  $D$ be a diagram of size $n$ and let $e$ be a prefix of $w_D$.
\\[3pt]
\begin{tabular}{rp{5.68in}}
(i) &
If $(a,b)$ and $(c,d)$ are nodes of $D$ satisfying $a<c$ and
$b\le d$, then $(a,b)t^D e<(c,d)t^D e$.
\\[3pt]
(ii) &
If $(a,b)$ and $(c,d)$ are nodes of $D$ satisfying $a<c$ and
$b\le d$ and $k=(a,b)t^{D}w_{J(\lambda_D)}$ and
$l=(c,d)t^{D}w_{J(\lambda_D)}$, then $k<l$ and
$kw_{J(\lambda_D)}w_D<lw_{J(\lambda_D)}w_D$.
\\[3pt]
(iii) &
If $k$ and $l$ are integers satisfying $1\le k<l\le n$ and
$(a,b)$ and $(c,d)$ are the nodes of $D$ at which the tableau
$t^{D}$ has the entries $k w_{J(\lambda_D)}$ and $l w_{J(\lambda_D)}$,
respectively, and if $kw_{J(\lambda_D)}w_D<lw_{J(\lambda_D)}w_D$,
then $a<c$ and $b\le d$.
\end{tabular}
\end{lemma}
\begin{proof}
(i)
Since $e$ is a prefix of $w_D$, $t=t^De$ is a standard $D$-tableau.
Hence, $(a,b)t<(c,d)t$.
That is, $(a,b)t^D e<(c,d)t^D e$.
\par
(ii)
Since $t^D$ is standard,
$kw_{J(\lambda_D)}=(a,b)t^D<(c,d)t^D=lw_{J(\lambda_D)}$.
As $kw_{J(\lambda_D)}$ is on a lower row of $t^D$ than
$lw_{J(\lambda_D)}$ and $w_{J(\lambda_D)}$ just permutes the entries
on each row of $t^D$, $k<l$.
Since $t_D=t^Dw_D$ is standard and $(a,b)t_D=kw_{J(\lambda_D)}w_D$
and $(c,d)t_D=lw_{J(\lambda_D)}w_D$,
$kw_{J(\lambda_D)}w_D<lw_{J(\lambda_D)}w_D$.
\par
(iii)
By hypothesis, $(a,b)t^D=kw_{J(\lambda_D)}$ and
$(c,d)t^D=lw_{J(\lambda_D)}$.
If $a=c$, then since $kw_{J(\lambda_D)}$ and $lw_{J(\lambda_D)}$ are
on the same row of $t^D$ and $w_{J(\lambda_D)}$ rearranges the
entries on each row of $t^D$ in decreasing order,
$lw_{J(\lambda_D)}<kw_{J(\lambda_D)}$.
As $t_D=t^Dw_D$ is a standard $D$-tableau,
$lw_{J(\lambda_D)}w_D<kw_{J(\lambda_D)}w_D$, contrary to hypothesis.
Hence, $a\neq c$.
\par
If $c<a$ then $lw_{J(\lambda_D)}<kw_{J(\lambda_D)}$.
Since the entries of $t^D$ are increasing by rows and $w_{J(\lambda_D)}$
only rearranges the entries on each row, $l<k$ contrary to hypothesis.
Hence, $a<c$.
\par
If $b>d$ then  $kw_{J(\lambda_D)}$ appears in a later column of $t^D$
than $lw_{J(\lambda_D)}$.
Hence, $kw_{J(\lambda_D)}w_D>lw_{J(\lambda_D)}w_D$, contrary to hypothesis.
So, $b\le d$.
\end{proof}

In view of the work in Schensted~\cite{Schensted1961} and Greene~\cite{Greene1974}, the preceding lemma motivates the following definition of a path in a diagram.

{\bf Definition.}
Let  $D$ be a diagram of size $n$.\\
(i) A \emph{path} of \emph{length} $m$ in $D$ is a sequence of nodes
$((a_i,b_i))_{i=1}^m$ of $D$ such that $a_{i}<a_{i+1}$ and
$b_i\le b_{i+1}$ for $i=1,\ldots,m-1$.
For $k\in\mathbb{N}$, a \emph{$k$-path} in $D$ is a sequence of $k$
mutually disjoint paths in $D$;
in particular, a $1$-path is a path.
The \emph{length} of a $k$-path is the sum of the lengths of its
constituent paths; this is the total number of nodes in the $k$-path.
\\
(ii) A $k$-path and a $k'$-path in $D$ are said to be \emph{equivalent} to one
another if they have the same set of nodes.
\\
(iii)
We say that $D$ is of \emph{subsequence type} $\nu$, where $\nu=(\nu_1,\ldots,\nu_r)\vdash n$, if the maximum length of a $k$-path in $D$ is
$\nu_1+\cdots+\nu_k$ whenever $1\le k\le r$.
(In particular, $D$ has an $r$-path containing all its nodes.)

\begin{remark}
\label{rem:3.6c}
Using the notion of \emph{$k$-increasing subsequence} of the row form
of a permutation (see, for example, \cite[Definition~3.5.1]{Sagan}),
we see from Lemma~\ref{lem:3.1c} that there is a bijection between the
set of $k$-paths in a diagram $D$ and the set of $k$-increasing
subsequences in $w_{J(\lambda_D)}w_D$, for any positive integer $k$.
\\
In particular,
the increasing subsequences occurring in the row-form of
$w_{J(\lambda_D)}w_D$ are precisely the ones which have form
$((a_i,b_i)t_D)_{i=1}^m$ for some path $((a_i,b_i))_{i=1}^m$ inside
$D$, where $(a_i,b_i)t_D$ denotes the entry of the tableau $t_D$ at
the node $(a_i,b_i)$.
\\
Also note that if $d\in \mathfrak{X}_{J(\lambda)}$ and the tableau
$t^Dd$ is standard, then the entries of $t^Dd$ appearing along paths
of $D$ give rise to increasing subsequences in the row-form of
$w_{J(\lambda_D)}d$ but in general, if $d\ne w_D$, there are other
increasing subsequences in the row-form of $w_{J(\lambda_D)}d$.
\end{remark}
\par
{\bf Example.} The diagram $D=
\begin{array}{llllll}
     .&      .& \times& \times&      .& \times\\
     .& \times&      .&      .& \times& \times\\
\times&      .&      .&      .&      .& \times\\
     .&      .&      .& \times&      .&      .\\
\end{array}
$ has maximal paths
$((3,1)$, $(4,4))$,
$((2,2)$, $(4,4))$,
$((2,2)$, $(3,6))$,
$((1,3)$, $(4,4))$,
$((1,3)$, $(2,5)$, $(3,6))$,
$((1,3)$, $(2,6)$, $(3,6))$,
$((1,4)$, $(2,5)$, $(3,6))$,
$((1,4)$, $(2,6)$, $(3,6))$,
$((1,4)$, $(4,4))$, and
$((1,6)$, $(2,6)$, $(3,6))$.
\par
In this case, $w_D=[3,4,7,2,6,8,1,9,5]$, $\lambda=(3,3,2,1)$
and $w_{J(\lambda)}w_D=[7,4,3,8,6,2,9,1,5]$.
The corresponding increasing subsequences in $w_{J(\lambda)}w_D$ are
$(1,5)$,
$(2,5)$,
$(2,9)$,
$(3,5)$,
$(3,6,9)$,
$(3,8,9)$,
$(4,6,9)$,
$(4,8,9)$,
$(4,5)$, and
$(7,8,9)$,
and may be read directly from the tableau $t^Dw_D$,
as $w_{J(\lambda)}w_D$ maps $t^Dw_{J(\lambda)}$ to $t^Dw_D$.
Thus,
\[
\begin{array}{llllll}
  &  & 3& 2&  & 1\\
  & 6&  &  & 5& 4\\
 8&  &  &  &  & 7\\
  &  &  & 9
\end{array}
\hspace{10pt}
\begin{array}{c}
\xrightarrow{w_{J(\lambda)}w_D}
\end{array}
\hspace{10pt}
\begin{array}{llllll}
  &  & 3& 4&  & 7\\
  & 2&  &  & 6& 8\\
 1&  &  &  &  & 9\\
  &  &  & 5
\end{array}
\]
The tableau
$
\begin{array}{llllll}
  &  & 2& 3&  & 4\\
  & 1&  &  & 6& 7\\
 5&  &  &  &  & 8\\
  &  &  & 9
\end{array}
$
corresponds to
$d$ $=$ $[2,3,4,1,6,7,5,8,9]$ and
\\
$w_{J(\lambda)}d$ $=$ $[4,3,2,7,6,1,8,5,9]$.
In addition to the increasing subsequences corresponding to the paths
in $D$,
$w_{J(\lambda)}d$ also has increasing subsequences such as
$(1,5)$ and $(2,7,8,9)$ which do not correspond to paths in $D$.
\par
We will be particularly interested in diagrams $D$ with subsequence
type $\lambda_D'$, which we call \emph{admissible} diagrams.
We make the following observation about paths in such diagrams.
\begin{proposition}
\label{prop:3.7c}
Let $D$ be a diagram and write
$\lambda_D'=(\lambda_1',\dots,\lambda_{r'}')$.
If $1\le u\le r'$, then a $u$-path in $D$ of length
$\sum_{1\le j\le u}\lambda_j'$ contains all $\lambda_i$ nodes on
the $i$-th row of $D$ if $\lambda_i\le u$ and exactly $u$ nodes on
all remaining rows.
\par
If for each $u$, $1\le u\le r'$, there is a $u$-path $\Pi_u$ such
that, for all $i$, $\Pi_u$ has exactly $\min\{u,\lambda_i\}$ nodes on
the $i$-th row, then $D$ is an admissible diagram.
\end{proposition}
\begin{proof}
Write $\lambda_D=(\lambda_1,\dots,\lambda_{r})$ and let
$E$ be a $(\lambda_D,\lambda_D')$-diagram.
Let $T_1=\{i\colon\lambda_i\ge u\}$ and $T_2=\{i\colon\lambda_i< u\}$.
By counting the nodes in the first $u$ columns of $E$, we get
$u|T_1|+\sum_{i\in T_2}\lambda_i=\sum_{1\le j\le u}\lambda_j'$.
Clearly, a $u$-path in $D$ has at most $\min\{\lambda_i,u\}$ nodes
on the $i$-th row of $D$, for all $i$.
If the $u$-path has exactly $\sum_{1\le j\le u}\lambda_j'$ nodes,
these inequalities must be exact.
\par
From the first part, we see that each $\Pi_u$ has length
$\sum_{1\le j\le u}\lambda_j'$.
Hence, $D$ is an admissible diagram.
\end{proof}
\par
We have the following simple bounds on the subsequence type of a
diagram.
\begin{proposition}
\label{prop:3.8c}
If $D$ is of subsequence type $\nu$ then
$\mu_D''\unlhd\nu\unlhd \lambda_D'$.
In particular, if $D$ is special then $\nu=\lambda_D'$.
\end{proposition}
\begin{proof}
Let $k\ge 1$. Since each column of $D$ gives a path in $D$, the $k$
columns in $D$ with the greatest number of nodes give a $k$-path of
length $\mu_{D,1}''+\dots+\mu_{D,k}''$.
Hence, $\mu_{D,1}''+\dots+\mu_{D,k}''\le \nu_1+\cdots+\nu_k$.
So, $\mu_D''\unlhd\nu$.
\par
Since the nodes of a path in $D$ are in different rows, there are no
more than $k$ nodes of a $k$-path in any row of $D$.
For a $k$-path of greatest length we get
$\nu_1+\cdots+\nu_k\le \lambda_{D,1}'+\dots+\lambda_{D,k}'$.
So, $\nu\unlhd\lambda_D'$.
\par
If $D$ is special then $\lambda_D'=\mu_D''$. Hence, $\nu=\lambda_D'$.
\end{proof}
We now relate the subsequence type of a diagram $D$ to the shape of
the Robinson-Schensted tableau of the element $w_{J(\lambda_D)}w_D$,
and establish a criterion for this element to be in the right cell
of $w_{J(\lambda_D)}$.
\begin{theorem}
\label{thm:3.9c}
Let $D$ be a diagram of size $n$ and let $\nu$ be a partition of $n$.
\\[3pt]
\begin{tabular}{rp{5.5in}}
(i) & $\shape (w_{J(\lambda_D)}w_D)=\nu$
if, and only if, $D$ is of subsequence type $\nu$.
\\[3pt]
(ii) & $w_{J(\lambda_D)}w_D\sim_R w_{J(\lambda_D)}$
if, and only if, $D$ is  of subsequence type $\lambda_D'$.
\end{tabular}
\end{theorem}
\begin{proof}
(i) Let $D$ be of subsequence type $\nu$ and let
$\shape (w_{J(\lambda_D)}w_D)=\tilde\nu$.
By Remark~\ref{rem:3.6c}, the maximum total length of a set of $k$
disjoint increasing subsequences in the row-form of
$w_{J(\lambda_D)}w_D$ is $\nu_1+\dots+\nu_k$, $k\ge1$.
From  \cite[Theorem~3.5.3]{Sagan},
$\shape (w_{J(\lambda_D)}w_D)=\nu$.
Hence, $\tilde\nu=\nu$.
\par
(ii) From (i), $D$ is of subsequence type $\lambda_D'$
if, and only if, $\shape (w_{J(\lambda_D)}w_D)=\lambda_D'$.
Since $\shape (w_{J(\lambda_D)})=\lambda_D'$,
we get $\shape (w_{J(\lambda_D)}w_D)=\lambda_D'$ if, and only if,
$w_{J(\lambda_D)}w_D)$ $\sim_{LR}$ $w_{J(\lambda_D)}$.
Since $w_{J(\lambda_D)}w_D$ $\LEQ_{R}$ $w_{J(\lambda_D)}$, we get
$w_{J(\lambda_D)}w_D$ $\sim_{LR}$ $w_{J(\lambda_D)}$
if, and only if,
$w_{J(\lambda_D)}w_D$ $\sim_{R}$ $w_{J(\lambda_D)}$,
by Lemma~\ref{res:2c}(iii).
\end{proof}
\begin{corollary}
\label{cor:3.10c}
If $D$ is a special diagram then
$w_{J(\lambda_D)}w_D\sim_R w_{J(\lambda_D)}$.
\end{corollary}
\begin{proof}
This follows from Proposition~\ref{prop:3.8c} and
Theorem~\ref{thm:3.9c}.
\end{proof}
\par
We can also deduce from Theorem~\ref{thm:3.9c} that
$\nu\unlhd\lambda_D'$, using the fact that
$w_{J(\lambda_D)}w_D\LEQ_R w_{J(\lambda_D)}$ and that
$\shape(x)\unlhd\shape(y)$ if $x,y\in S_n$ and $x\LEQ_R y$
by \cite[Theorem~5.1]{Gec05} from which it follows that
$\nu$
$=\shape(w_{J(\lambda_D)}w_D)\unlhd\shape(w_{J(\lambda_D)})$
$=\lambda_D'$.
\par
Note that Theorem~\ref{thm:3.9c} (ii) is trivially equivalent to
``$w_D\in Z(\lambda_D)$ if, and only if,
$D$ is  of subsequence type $\lambda_D'$''.
In particular, $w_D\in Z(\lambda_D)$ whenever $D$ is special.
\par
We say that a diagram $D$ is {\it admissible}
if it is of subsequence type $\lambda_D'$.
\par
{\bf Example.} Consider
$D=
\begin{array}{llll}
      &\times&\times&\times\\
\times&\times&      &      \\
      &      &\times&      \\
\end{array}$.
Then $\mu_D''=(2,2,1,1)$, $\lambda_D'=(3,2,1)$ and
$D$ is of subsequence type $(3,1,1,1)$.
Consequently, $w_D\not\in Z(\lambda_D)$ since $D$ is not admissible.
%
%
%
\begin{proposition}
\label{prop:3.12c}
Let $\lambda\vDash n$ and let $D\in\mathcal{D}^{(\lambda)}$ be an
admissible diagram.
Then $w_D\notin Y(\lambda)$ if, and only if, there exists an
admissible diagram $E\in\mathcal{D}^{(\lambda)}$
such that $w_D\neq w_E$ and $t^Ew_D$ is a standard $E$-tableau.
\end{proposition}
\begin{proof}
By hypothesis and Theorem~\ref{thm:3.9c} (ii), $w_D\in Z(\lambda)$.
\par
For the `if' part:
In this case,
Since $t^Ew_D$ is a standard $E$-tableau, $w_D$ is a prefix of $w_E$.
Moreover, $w_E\in Z(\lambda)$ since $E$ is admissible.
Since $w_D\neq w_E$, $w_D\not\in Y(\lambda)$.
\par
For the `only if' part:
Since $w_D\notin Y(\lambda)$, there is an $s\in S$ such that
$l(w_Ds)=l(w_D)+1$ and $w_Ds\in Z(\lambda)$.
Let $E=D(w_Ds,\lambda)$.
Clearly, $E$ is admissible and $w_E=w_Ds\neq w_D$.
Since $w_D$ is a prefix of $w_E$, $t^Ew_D$
is a standard $E$-tableau.
This completes the proof.
\end{proof}
\par
{\bf Example.} Consider the diagrams $D$ and $E$ and the standard $E$-tableau $t$
where\newline
$D=\begin{array}{lll}
      &\times&      \\
\times&\times&      \\
\times&\times&\times\\
      &\times&
\end{array}$
and
$E=
\begin{array}{llll}
      &      &\times&      \\
      &\times&\times&      \\
\times&      &\times&\times\\
      &      &\times&      \\
\end{array}$
and
$t=
\begin{array}{llll}
 & &3& \\
 &1&4& \\
2& &5&7\\
 & &6& \\
\end{array}$.
\newline
We then get that $w_D \in Z(\lambda_D)\setminus Y(\lambda_D)$
in view of Proposition~\ref{prop:3.12c}
(note that $\lambda_D=\lambda_E$ and both $D$ and $E$ are of
subsequence type $(4,2,1)=\lambda_D'$).

The \emph{reverse} composition $\dot\lambda$ of a composition
$\lambda=(\lambda_1$, $\dots$, $\lambda_r)$ is the composition
$(\lambda_r$, $\dots$, $\lambda_1)$ obtained by reversing the order of
the entries.
For a principal diagram $D\in\mathcal{D}^{(\lambda)}$,
the diagram $\dot D\in\mathcal{D}^{(\dot\lambda)}$ is the diagram
obtained by rotating $D$ through $180^{\circ}$.
If $D\in\mathcal{D}^{(\lambda,\mu)}$,
then $\dot D\in\mathcal{D}^{(\dot\lambda,\dot\mu)}$.
Since rotating $D$ through $180^{\circ}$ maps $k$-paths into $k$-paths,
for any $k$, $D$ and $\dot D$ have the same subsequence type.
\begin{proposition}
\label{thm:3.13c}
Let $\lambda\vDash n$. Then there is a natural bijection
$Y(\lambda)\rightarrow Y(\dot\lambda)$ given by $y\mapsto w_0 y w_0$,
where $w_0$ is the longest element in $W$, which extends to a
bijection $Z(\lambda)\rightarrow Z(\dot\lambda)$ and restricts to a
bijection $Y_s(\lambda)\rightarrow Y_s(\dot\lambda)$.
Moreover, $\mathfrak{C}(\dot\lambda)=w_0\mathfrak{C}(\lambda)w_0$.
\end{proposition}
\begin{proof}
Let $y\in Y(\lambda)$ and write $y=w_D$ where
$D\in\mathcal{D}^{(\lambda)}$.
Then $D$ is admissible by Theorem~\ref{thm:3.9c}(ii).
Hence, $\dot D\in\mathcal{D}^{(\dot\lambda)}$ is admissible.
Let $\dot y=w_{\dot D}$.
Then $\dot y\in Z(\dot\lambda)$ by Theorem~\ref{thm:3.9c}(ii).
Also, $\dot y=w_0 y w_0$.
If $\dot y\not\in Y(\dot\lambda)$, then $\dot y$ is a proper prefix of
some $z\in Y(\dot\lambda)$.
So, $y$ is a proper prefix of $\dot z$, where $\dot z=w_0zw_0$.
By the preceding argument, $\dot z\in Z(\lambda)$.
This contradicts the fact that $y\in Y(\lambda)$.
Hence, $\dot y\in Y(\dot \lambda)$.

If $w$ is an arbitrary prefix of $y$, it is immediate that $w_0ww_0$
is a prefix of $\dot y$.
Since $Z(\lambda)$ is the set of all prefixes of elements of
$Y(\lambda)$, the second bijection is established.
The third bijection comes from the fact that the diagram $D$ is
special if, and only if, $\dot D$ is special.

Finally, $\mathfrak{C}(\dot\lambda)$
$=w_{J(\dot\lambda)}Z(\dot\lambda)$
$=w_0w_{J(\lambda)}w_0w_0Z(\lambda)w_0$
$=w_0\mathfrak{C}(\lambda)w_0$.
\end{proof}

Note that if $\rho$ is the representation of $S_n$ corresponding to
the cell $\mathfrak{C}(\lambda)$, then the representation corresponding
to $\mathfrak{C}(\dot\lambda)$ is given by $s_i\mapsto s_{n-i}\rho$
for all $i$.

\end{section}
%

%
%

When $\lambda$ is a partition we have seen in Remark~\ref{ypartition} that $Y(\lambda)=\{w_D\}$ where $D$ is the Young diagram corresponding to $\lambda$.
It then follows from Proposition~\ref{thm:3.13c} that in the case of a composition~$\mu$ such that $\dot\mu$ is a partition, we have $Y(\mu)=\{w_0w_Ew_0\}$, where $E$ is the Young diagram corresponding to $\dot\mu$.

However, determining $Y(\lambda)$ for an arbitrary composition $\lambda$ turns out to be much more complicated in general.
We conclude this section with two propositions in which $Y(\lambda)$ is determined for two families of compositions.
We will first deal with compositions which are rearrangements of hook partitions and in this case special diagrams will turn out to play an important part. 
Recall that if $\lambda$, $\mu$ are compositions of $n$ with $\mu''=\lambda'$, then there is a unique diagram $D$ with $\lambda_D=\lambda$ and $\mu_D=\mu$ (such a diagram $D$ is special).

\begin{proposition}
\label{thm:5.1c}
Let $\lambda\vDash n$ and suppose $\lambda$ is a rearrangement of the hook partition $m^11^{r-1}$ where $m>1$ and $r\ge3$.
Suppose further that none of $\lambda$ or $\dot\lambda$ is a partition.
Then $Y(\lambda)=\{w_D:$ $D\in\mathcal{D}^{(\lambda)}$ and $\mu''_D=\lambda'\}$.
In particular, $Y(\lambda)=Y_s(\lambda)$ and $|Y(\lambda)|=m$.
\end{proposition}
\begin{proof}
Assume the hypothesis.
Then $\lambda=(\lambda_1,\ldots,\lambda_r)$ where $\lambda_k=m$ for some $k$ with $1<k<r$ and $\lambda_i=1$ for all
$i\in\{1,\ldots,r\}\setminus\{k\}$.
Let $y\in Y(\lambda)$ and let $D=D(y,\lambda)$.
It follows that $D$ has a path $\Pi$ of length $r$ which necessarily
has a node on each row of $D$.
Let $(i,b_i)$, $1\le i\le r$, be the nodes of $\Pi$.
From the definition of path, $b_i\le b_{i+1}$ for $1\le i\le r-1$.
Now suppose that $b_l<b_{l+1}$ for some $l\in\{1,\ldots,r-1\}$.
If $l\ge k$ (resp. $l<k$) consider the diagram $E$ obtained from~$D$
by removing the node in position $(l+1,b_{l+1})$, (resp. $(l,b_l)$)
and introducing a node in position $(l+1,b_l)$, (resp. $(l,b_{l+1})$).
Clearly, $\lambda_E=\lambda$.
Since all nodes of $D$ which are not on the $k$-th row are on the path
$\Pi$, $t^Ew_D$ is standard.
The subsequence type $\nu$ of $E$ has $r$ as its first entry and
satisfies $\nu\unlhd\lambda'$ by Proposition~\ref{prop:3.8c}.
Hence, $\nu=\lambda'$ (that is, $E$ is admissible).
Clearly, $w_E\ne w_D$ so by Proposition~\ref{prop:3.12c}, $w_D\not\in Y(\lambda)$, contrary to hypothesis.
Hence, $b_i=b_{i+1}$ for $1\le i\le r-1$.
We conclude that $D$ is a special diagram.

For each $\lambda$ there are $m\ (=n-r+1)$ principal special diagrams.
Moreover, our hypothesis that $1<k<r$ ensures that the corresponding $w_D$
are not prefixes of one another by Result~\ref{res:7c}.
The required result now follows easily.
\end{proof}

We remark here that Proposition~\ref{thm:5.1c} follows from~\cite[Proposition 5.2]{MPa15}.
However, its proof is done in the spirit of the techniques developed in this paper.

\bigskip
Next we deal with a family of compositions $\lambda=\lambda(r)$ for which, as it turns out, $Y(\lambda)\ne Y_s(\lambda)$ for $r\ge 4$.

\begin{proposition}\label{prop:3.10n}
Let $r\ge 3$ and let $\lambda=(2,1^{r-2},2)\vDash n=r+2$.
Then $Y(\lambda)=\{w_{D(a)}:\ 2\le a\le r\}$ where, for $2\le a\le r$, $D(a)$ is the diagram defined by $D(a)=\{(i,1):$ $1\le i\le a-1\}$ $\cup$ $\{(i,2):$ $a\le i\le r\}$ $\cup$ $\{(r,1),\, (1,2)\}$.
\end{proposition}

\begin{proof}
Let $\lambda$, $r$, $a$ and $D(a)$ be as in the statement of the proposition.
\begin{center}
$ \begin{array}{ccccc}
D(a) & = &
\begin{array}{cc}
 \times & \times  \\[-5pt]
 \vdots &         \\
 \times &         \\
        & \times  \\[-5pt]
        & \vdots  \\
 \times & \times  \\
 \end{array}
\end{array}$
\end{center}
It is easy to observe from Result~\ref{res:7c} that for $i,j$ with $2\le i,j\le r$ and $i\ne j$, we have that $w_{D(i)}$ cannot be a prefix of $w_{D(j)}$.
It is thus sufficient to show that any element of $Z(\lambda)$ is a prefix of $w_{D(a)}$ for some $a$ with $2\le a\le r$.

Let $z\in Z(\lambda)$ be given and set $D=D(z,\lambda)$.
Then $D$ is admissible by Theorem~\ref{thm:3.9c}.
Suppose $D=\{(i,j_i):$ $1\le i\le r\}$ $\cup$ $\{(i,j_1'),$ $(r,j_r')\}$ where $j_1<j_1'$ and $j_r<j_r'$.
Since $D$ is of subsequence type $\lambda'=(r,2)$, it has a path $\Pi$ of length $r$ and two disjoint paths $\Pi_1$ and $\Pi_2$ which contain all nodes of $D$.
Necessarily, $\Pi$ contains the nodes $(i,j_i)$ with $2\le i\le r-1$ and precisely one node from each of the remaining rows.
It follows that $j_1\le j_2\le\ldots\le j_{r-1}\le j_r'$.
Moreover, we must have that $\Pi_1$ contains precisely one node from each of the first and last rows of $D$ while $\Pi_2$ contains the remaining two nodes from these two rows.
This forces $j_1\le j_r$ and $j_1'\le j_r'$.

If $j_1'\le j_2$, it is then an easy consequence of Result~\ref{res:7c} that $z\,(=w_D)$ is a prefix of $w_{D(2)}$.
Similarly, if $j_{r-1}\le j_r$ we can see that $z$ is a prefix of $w_{D(r)}$.


It remains to consider the case when $j_1'>j_2$ and $j_{r-1}>j_r$.
First, we show that if $i$ satisfies $2\le i\le r-1$,
then either $j_i\le j_r$ or $j_1'\le j_i$.
Suppose, on the contrary, that $j_i>j_r$ and $j_1'>j_i$ for some $i$
with $2\le i\le r-1$.
Then $j_1'>j_i>j_r$ for some $i$ with $2\le i\le r-1$.
Hence the nodes $(1,j_1')$ and $(r,j_r)$ cannot both belong to the
same one of the paths $\Pi_1$ or $\Pi_2$ described above.
It follows that $(1,j_1)$ and $(r,j_r)$ belong to the same path,
say $\Pi_1$, and $(1,j_1')$, $(r,j_r')$ belong to $\Pi_2$.
But $(i,j_i)\not\in\Pi_1$ since $j_i>j_r$ and $(i,j_i)\not\in\Pi_2$
since $j_i<j_1'$.
This contradicts the fact that $D=\Pi_1\cup\Pi_2$ and shows that
either $j_i\le j_r$ or $j_1'\le j_i$ if $2\le i\le r-1$.
In particular, taking $i=r-1$, we get $j_1'\le j_{r-1}$.

Now let $a$ be the smallest integer with $2\le a\le r$ such that $j_a\ge j_1'$.
Note that $2<a<r$ since $j_1'>j_2$ by our assumption, and $j_1'\le j_{r-1}$, as we have seen already.
It follows that $j_{a-1}<j_1'$ and hence $j_{a-1}\le j_r$ in view of our observation above.
This ensures that $\Pi_1'=\{(1,j_1), (2,j_2),\ldots, (a-1,j_{a-1}),(r,j_r)\}$ and $\Pi_2'=\{(1,j_1'), (a,j_a),\ldots, (r,j_r')\}$ are both paths in $D$ and hence the tableau $t^{D(a)}z$ is standard.
Invoking Result~\ref{res:7c}, we conclude that $z$ is a prefix of $w_{D(a)}$.
\end{proof}


\begin{section}
{Lifting cells}
\label{sec:6c}
In this section, we examine how the cell associated with a composition
$\lambda$ of $n$ may be related to the cells associated with certain
related compositions of $n+1$.

\subsection{A process relating to the induction of cells}\label{SubsectionProcessInducCells}

The \emph{lower star} composition $\lambda_*$ of a composition
$\lambda=(\lambda_1,\dots,\lambda_r)$ of $n$ is the composition
$(\lambda_1,\dots,\lambda_r,1)$ of $n+1$.
\par
If $\lambda=(1,2,1,2)$,
then $Y(\lambda)$ $=$ $\{y_1,y_2\}$,
where $y_1$ $=$ $[3,1,4,5,2,6]$ and $y_2$ $=$ $[1,2,5,3,4,6]$,
and $Y(\lambda_*)$ $=$ $\{y_1',y_2',y_3'\}$, where
$y_1'\!$ $=$ $\![ 1$, $2$, $6$, $3$, $4$, $7$, $5 ]$,
$y_2'\!$ $=$ $\![ 3$, $1$, $4$, $5$, $2$, $6$, $7 ]$,
and
$y_3'\!$ $=$ $\![ 2$, $1$, $3$, $4$, $5$, $7$, $6 ]$.
%
\begin{table}[h]
\begin{center}
$\begin{array}{ccccccccc}
\begin{array}[t]{cc}
      & 3 \\
    1 & 4 \\
      & 5 \\
    2 & 6 \\
\end{array}
&
\hspace{10pt}
&
\begin{array}[t]{cc}
    1    \\
    2 & 5 \\
    3    \\
    4 & 6 \\
\end{array}
&
\hspace{10pt}
&
\begin{array}[t]{cc}
    1    \\
    2 & 6 \\
    3    \\
    4 & 7 \\
    5    \\
\end{array}
&
\hspace{10pt}
&
\begin{array}[t]{cc}
      & 3 \\
    1 & 4 \\
      & 5 \\
    2 & 6 \\
      & 7 \\
\end{array}
&
\hspace{10pt}
&
\begin{array}[t]{ccc}
      & 2    \\
    1 & 3    \\
      & 4    \\
      & 5 & 7 \\
      & 6    \\
\end{array}
\\
\\[-7pt]
y_1 && y_2 && y_1' && y_2' && y_3'
\end{array}$
\end{center}
\caption{The tableaux of the elements of $Y(\lambda)$ and
$Y(\lambda_*)$, $\lambda=(1,2,1,2)$.}
\label{tbl:4c}
\end{table}
\par
Extending the diagram underlying $y_1$ to an admissible diagram in
$\mathcal{D}^{(1,2,1,2,1)}$ by placing a node in the fifth row
requires the node to be in the second column (or later),
whereas the diagram underlying $y_2$ can be so extended by placing
the node anywhere in the fifth row.

\medskip
{\bf Definition.}
Let $\lambda\vDash n$ and let $D\in \mathcal D^{(\lambda)}$.
For $1\le j\le n$, let $u_j$ be the node $u$ of $D$ for which $ut_D=j$ (so we have $(u_j)t_D=j$).
Suppose further that $u_j=(a_j,b_j)$ for $1\le j\le n$.

(i) Define $D_0'\in\mathcal D^{(\lambda_*)}$ by $D_0'=\{(i,j+1):\ (i,j)\in D\}\cup \{(r_D+1,1)\}$.

(ii) For $1\le i\le c_D$, define $D_i'\in\mathcal D^{(\lambda_*)}$ by $D_i'=\{(a,b):\ (a,b)\in D\}\cup\{(r_D+1,i)\}$.

(iii) For each $i$, with $1\le i\le n$, define $D'(u_i)\in\mathcal D^{(\lambda_*)}$ as follows:\\
(a) If $b_{i+1}>b_i$, then $D'(u_i)=D\cup \{(r_D+1,b_i)\}$\\
(b) If $b_{i+1}=b_i$, then $D'(u_i)=\{u_j\in D:\ j\le i\}\cup \{(r_D+1,b_i)\}\cup\{(a_j,b_{j}+1):$ $(a_j,b_j)=u_j\in D$ and $j>i\}$

Informally, for item~(iii), if $u_i=(a_i,b_i)$ is the last node in its column, $D'(u_i)$ is obtained from $D$ by introducing a node in the same column as $u_i$ (column $b_i$) in a row immediately after the last row of $D$.
Otherwise, $D'(u_i)$ is obtained by inserting a new column to $D$ immediately after column $b_i$ and moving into this new column all nodes in column $b_i$ below node $u_i$ while introducing an extra node in column $b_i$ in a row immediately after the last row of~$D$.

\medskip
Now let $D\in \mathcal D^{(\lambda)}$ where $\lambda\vDash n$.
It is clear from the above definition that $D_0'$ is not admissible even in the case $D$ is assumed to be admissible.
On the other hand, if $D$ is admissible then $D'(u_n)$ is also admissible.
This is because any $k$-path of length $l$ in $D$ can be converted to a $k$-path of length $l+1$ in $D'(u_n)$ by just adding to it the single node in the last row of $D'(u_n)$.
It is also easy to observe that for $j$, with $1\le j\le c_D$, we have that $D_j'=D'(u_t)$ where $u_t$ is the last node in column $j$ of $D$; in particular if $s=c_D$ then $D_s'=D'(u_n)$.

An immediate consequence of the above is that whenever $D$ is an admissible (principal) diagram of size $n$ then both sets $\{i:\ 1\le i\le n$ and $D'(u_i)$ is admissible$\}$ and $\{j:\ 1\le j\le c_D$ and $D_j'$ is admissible$\}$ are nonempty.
This leads to the following definition.

{\bf Definition.}
Let $D$ be an admissible (principal) diagram of size $n$.
We set $p(D)=\min\{i:\ 1\le i\le n$ and $D'(u_i)$ is admissible$\}$ and $q(D)=\min\{j:\ 1\le j\le c_D$ and $D_j'$ is admissible$\}$.

\begin{remark}\label{RemarkwD'}

Let $\lambda\vDash n$ and let $D\in \mathcal D^{(\lambda)}$ be admissible.

(i)
It then follows that $D'(u_i)$ is admissible for $i\ge p(D)$ and $D'(u_i)$ is not admissible for $i<p(D)$.
This is because for $j<l$, with $1\le j,l\le n$, we have that for any $k$-path in $D'(u_j)$ there is a `corresponding' $k$-path of the same length in $D'(u_l)$.

(ii) Fix $i$ with $1\le i\le n$ and set $D'=D'(u_i)$.
An easy computation shows that $w_{D'}=w_D(i+1,\, i+2,\, \ldots,\, n+1)=w_Ds_ns_{n-1}\cdots s_{i+1}$.
Hence, for $j>i$ we have, on setting $E'=D'(u_j)$, that $w_{E'}$ is a prefix of $w_{D'}$ which is in agreement with the above observation that $w_{E'}\in Z(\lambda_*)$ whenever $w_{D'}\in Z(\lambda_*)$.

On the other hand, if we set $D'=D_0'$ it is also easy to observe that $w_{D'}=w_Ds_ns_{n-1}\ldots s_2s_1$.

We can deduce that $\mathfrak{X}'=\{w_D^{-1}w_{D'}:$ $D'=D_0'$ or $D'=D'(u_i)$ for some $i$ with $1\le i\le n\}$.
\end{remark}

\begin{lemma}\label{LemmaUpIsTheLastNode}
Let $\lambda\vDash n$ and let $D\in \mathcal D^{(\lambda)}$ be an admissible (principal) diagram.
Also let $p=p(D)$ and $q=q(D)$.
We then have that $u_p$ is the last node in column $q$ of $D$ (and hence $D'(u_p)=D'_q$).
\end{lemma}

\begin{proof}
We assume the hypothesis and suppose that
$\lambda'=(\lambda_1',\ldots,\lambda_{r'}')$.
For each $k$ with $1\le k\le r'$, let $\mathfrak{P}_k$ be the set of
$k$-paths in $D$ of length $\lambda_1'+\ldots+\lambda_{k}'$
(recall that $\mathfrak{P}_k$ is nonempty for $1\le k\le r'$ since $D$
is admissible).
For each $\Pi\in\mathfrak{P}_k$, let $m_\Pi=\min\{i:$ there exists a
path in $\Pi$ that terminates at $u_i\}$.
Also let $m_k(D)=\min\{m_\Pi:\ \Pi\in\mathfrak{P}_k\}$.

For the moment, we fix $k$.
We write $l=m_k(D)$.
We claim that $u_l$ is the last node in its column.
This is immediate for $k=1$, since $u_l$ is at the end of a path
of length $\lambda_1'$ in this case.
So we need only consider $1<k\le r'$.
Also observe that what is stated in the claim is trivially true when $l=n$.

Assume now that $l<n$ and that $u_l$ and $u_{l+1}$ are in the same column of $D$ (that is, $u_l$ is not the last node in its column)
and let $\Pi'$ be a $k$-path in $D$ of length
$\lambda_1'+\ldots+\lambda_{k}'$ such that $m_{\Pi'}=l$.
Let $\Pi_1$ be a path in $\Pi'$ that terminates at $u_l$ and let $l_1$ be the length of $\Pi_1$.
Then $\Pi_1'=\Pi_1\cup\{u_{l+1}\}$ is
a path in $D$ of length $l_1+1$.
If $u_{l+1}\not\in\Pi'$ then $(\Pi'\setminus\Pi_1)\cup\Pi_1'$ forms a
$k$-path in $D$ of length $\lambda_1'+\ldots+\lambda_{k}'+1$.
This is impossible since $D$ is of subsequence type
$(\lambda_1',\ldots,\lambda_{r'}')$.
Hence, $u_{l+1}\in\Pi'$.

Since $u_l$ is the last node in path $\Pi_1$ of $\Pi'$, $u_{l+1}$ is
in a different path $\Pi_2$ $(\ne\Pi_1)$ of $\Pi'$.
Let $\overline{\Pi}_2=\{u_i\in\Pi_2:\ i\ge l+1\}$.
Then $\Pi_1\cup\overline{\Pi}_2$ forms a path in $D$ and its nodes
are in $\Pi'$.
If $\Pi_2\subseteq\overline{\Pi}_2$, it would then follow
that $\Pi'$ is equivalent to a $(k-1)$-path of length
$\lambda_1'+\ldots+\lambda_{k}'$ contrary to the admissibility of $D$
since $\lambda_k'\ge 1$.
Hence $\Pi_2\setminus\overline{\Pi}_2\ne\varnothing$ and
$\Pi_2\setminus\overline{\Pi}_2$ is a path terminating at node $u_{l'}$
where $l'<l$.
Considering the decomposition
$(\Pi'\setminus(\Pi_1\cup\Pi_2))\cup(\Pi_1\cup\overline{\Pi}_2)\cup(
\Pi_2\setminus\overline{\Pi}_2)$ we see that $\Pi'$ is equivalent to a
$k$-path $\Pi''$ in $D$ where $m_{\Pi''}\le l'<l$.
This contradicts the choice of $l$ and establishes that $u_{l}$ and
$u_{l+1}$ are in different columns of $D$.
In particular, $u_{l}$ is the last node in its column.

Now we define $m(D)=\max\{m_k(D):\ 1\le k\le r'\}$ and let $m=m(D)$.
It is clear from this definition that $u_m$ is the last node in its
column.
Moreover, the above construction ensures that $D'(u_i)$ is admissible
for $i\ge m$ and $D'(u_i)$ is not admissible for $i<m$.
We conclude that $m(D)=p(D)$.
The required result now follows easily.
\end{proof}

We are now ready to state the main result of this subsection.

\begin{theorem}\label{TheoremReplaceThm6.1}
Let $\lambda\vDash n$. For $z\in Z(\lambda)$ set $D=D(z,\lambda)$ and $D'=D_q'$, where $q=q(D)$.
Then there is a well-defined injective map $\theta_*$ from $Z(\lambda)$ to $Z(\lambda_*)$ in which $z\mapsto z'$, where $z'=w_{D'}$.
Moreover, $\theta_*$ has the following properties.

(i) For $z\in Z(\lambda)$ we have $z^{-1}(z\theta_*)\in\mathfrak{X}'$ and whenever $zx\in Z(\lambda_*)$ with $x\in\mathfrak{X}'$ we have that $x$ is a prefix of $z^{-1}(z\theta_*)$.

(ii) $Y(\lambda)\theta_*\subseteq Y(\lambda_*)\subseteq Z(\lambda)\theta_*$.
\end{theorem}

\begin{proof}
It is clear that the map $\theta_*$ from $Z(\lambda)$ to $Z(\lambda_*)$ described in the statement of the theorem is well-defined.
Moreover, with $z$, $D$, $D'$ and $z'$ as above, it is easy to observe using Result~\ref{res:ischia} that diagram $D'$ is in fact exactly the same as diagram $D(z',\lambda_*)$.
Suppose now that $z_1, z_2\in Z(\lambda)$ satisfy $z_1\theta_*=z'=z_2\theta_*$ for some $z'$ in $Z(\lambda_*)$.
Now $z'$ together with $\lambda_*$ determine uniquely diagram $D'$ (this is because $D'=D(z',\lambda_*)$).
But the construction via which $\theta_*$ is defined ensures that, by removing the single node in its last row, $D'$ determines diagrams $D(z_1,\lambda)$ and $D(z_2,\lambda)$ uniquely, forcing them to be equal.
The equality $D(z_1,\lambda)=D(z_2,\lambda)$ ensures that $z_1=z_2$ and this is enough to establish that $\theta_*$ is injective.

For item (i), let $z\in Z(\lambda)$ and let $D=D(z,\lambda)$.
Then $z\theta_*=w_{D'}$ where $D'=D_q'$ ($q=q(D)$) by the definition of $\theta_*$.
We know from Lemma~\ref{LemmaUpIsTheLastNode} that $D_q'=D'(u_p)$ where $p=p(D)$.
It is now immediate from Remark~\ref{RemarkwD'} that any $x\in\mathfrak{X}'$ which satisfies $zx\in Z(\lambda_*)$ must be of the form $z^{-1}w_E$ where $E=D'(u_j)$ for some $j$ with $j\ge p$ and, in addition, that such an $x$ must be a prefix of $z^{-1}(z\theta_*)$.

In order to establish the second inclusion of sets in item (ii), suppose $\overline{y}\in Y(\lambda_*)$ and consider diagram $C=D(\overline{y},\lambda_*)$.
Then $C\in\mathcal D^{(\lambda_*)}$ and $C$ is an admissible diagram.
Now let $B\in \mathcal D^{(\lambda)}$ be the diagram obtained from $C$ by removing the single node in its last row.
Let $\nu$ be the subsequence type of $B$.
Then $\nu\unlhd\lambda'=(\lambda_1',\ldots,\lambda_{r'}')$ by
Proposition~\ref{prop:3.8c}.
As $C$ is of subsequence type
$(\lambda_1'+1,\lambda_2'\ldots,\lambda_{r'}')=\lambda_*'$,
if $1\le k\le r'$ then any $k$-path of length
$\lambda_1'+\cdots+\lambda_{k}'+1$ in $C$ gives a $k$-path of length
$\ge\lambda_1'+\cdots+\lambda_{k}'$ in $B$.
So $\lambda'\unlhd\nu$.
Thus $B$ is admissible.

Now let $F=D(w_B,\lambda)$.
Then $w_F=w_B\in Z(\lambda)$ as $B$ is admissible.
From the construction of $B$, either $F=B$ or $F$ is obtained from $B$ by
merging two adjacent columns (see Result~\ref{res:ischia}).
Since $C$ is admissible and $C=D(\overline{y},\lambda_*)$, we have $C=F'(u_i)$ for some $i$ with $p(F)\le i\le n$ (see Remark~\ref{RemarkwD'}).
Since $\overline{y}\in Y(\lambda_*)$, $i=p(F)$ again by Remark~\ref{RemarkwD'}.
Hence $C=F'_{q(F)}$ by Lemma~\ref{LemmaUpIsTheLastNode}.
We deduce that $\overline{y}\,(=w_C)=w_F\theta_*\in Z(\lambda)\theta_*$ as required.

Finally we want to establish that $Y(\lambda)\theta_*\subseteq Y(\lambda_*)$.
For this, let $y\in Y(\lambda)$ and suppose that $y\theta_*\not\in Y(\lambda_*)$.
Let $y'=y\theta_*$.
Also let $D=D(y,\lambda)$ and $D'=D(y',\lambda_*)$.
Then $y'=w_{D'}$ and $D'=D'(u_p)=D_q'$ where $p=p(D)$ and $q=q(D)$.
It follows that $(n+1)y'=p+1$.
Since $y'\not\in Y(\lambda_*)$, by Proposition~\ref{prop:3.12c} there exists an admissible (principal) diagram $E\in\mathcal D^{(\lambda_*)}$ such that $t^Ey'$ is standard and $w_E\ne w_{D'}$.
We can assume without loss of generality that $E=D(w_E,\lambda_*)$
since, if $E'=D(w_E,\lambda_*)$, then $w_E=w_{E'}$ and $t^{E'}y'$
is standard from Result~\ref{res:ischia}.
Now let $F$ be the diagram obtained from $E$ by removing the single node on its last row.
Then $F$ is admissible by similar argument as above.
Next, we would like to compare the tableaux $t^Ey'$ and $t^Fy$. 
For this, define subsets $T_1$ and $T_2$ of $\{1,\ldots,n\}$ as follows :
$T_1=\{i:\ 1\le i\le n$ and $iy'<p+1\}$ and $T_2=\{i:\ 1\le i\le n$ and $iy'>p+1\}$.
We then have $iy=iy'$ for $i\in T_1$ and $iy=(iy')-1$ for $i\in T_2$.
Thus, tableau $t^Fy$ is obtained from $t^Ey'$ by removing the single entry in the last row of $t^Ey'$ and replacing each number $l>p+1$ by $l-1$.
It follows that $t^Fy$ is standard since $t^Ey'$ is standard.
By Result~\ref{res:7c}, $y$ is a prefix of $w_F$.
Since $w_F\in Z(\lambda)$ as $F$ is admissible and $y\in Y(\lambda)$,
we get $w_F=y\,(=w_D)$.
At this point it is also useful to observe that the single node in the last row of $E\,(=D(w_E,\lambda_*))$  cannot be the sole node of $E$ in its column.
For this, let the single node in row $r_E+1$ of $E$ be in column $j$.
That this node cannot be the sole node in its column follows in the case $j=1$ from Remark~\ref{RemarkwD'} since $E$ is admissible, and in the case $j>1$ from the fact that $E=D(w_E,\lambda_*)$.

Combining the above observations with  Result~\ref{res:ischia} we see that either $F=D$ or $F$ can be obtained from $D$ by splitting a column of $D$ into two successive columns so that the nodes in the column with lesser column index have row indices which are less than the row indices of all the nodes in the column with greater column index.

We conclude that $E=D'(u_a)$ for some $a$.
Since $E$ is admissible, $a\ge p$.
Since $D'=D'(u_p)$ and $E\ne D'$, $a\ne p$.
It follows from this that the entry $p+1$ which appears in the last row of $t^Ey'$ must lie in a column which includes entries greater than $p+1$ since $p+1$ is not the sole entry in its column.
This contradicts the fact that $t^Ey'$ is standard.
Hence $y'\in Y(\lambda_*)$.
\end{proof}

We illustrate Theorem~\ref{TheoremReplaceThm6.1} with the following example.
Consider $\lambda=(2,1,5)$.
Then $\lambda_*=(2,1,5,1)$.
$Y(\lambda)=\{y_1,y_2\}$, where $y_1=(1,4)(2,6,3,7,5)$ and
$y_2=(1,4)(2,7,6,3,5)$, and
$Y(\lambda_*)=\{y_1',y_2',y_3',y_4',y_5'\}$, where
$y_1'=(1,4)(2,6,3,7,5)$, $y_2'=(1,4)(2,8,9,7,6,3,5)$,
$y_3'=(1,3,4)(2,8,9,6,5)$, $y_4'=(1,2,8,9,5,4)$, and
$y_5'=(2,8,9,4,3)$.
Putting $x_1=1$ and $x_2=(7,8,9)=s_8s_7$, we see that $y_1'=y_1x_1$,
$y_2'=y_2x_2$, and $x_1,x_2\in\mathfrak{X}'$.
\par

\begin{remark}
(i) If in the proof for the second inclusion of sets in item~(ii) of Theorem~\ref{TheoremReplaceThm6.1} we assume in place of ``$\overline{y}\in Y(\lambda_*)$'' that ``$\overline{y}\in Z(\lambda_*)$'' instead, the proof goes through without any change up to the point where we deduce that there exists $i$ with $p(F)\le i\le n$ such that $C=F'(u_i)$.
It follows from this that $\overline{y}=w_C=w_Fs_n\ldots s_{i+1}=zs_n\ldots s_{i+1}$ where $z\in Z(\lambda)$ (see Remark~\ref{RemarkwD'}).
Since $w_{J(\lambda_*)}=w_{J(\lambda)}$, we deduce that $\mathfrak C(\lambda_*)\subseteq \mathfrak C(\lambda)\mathfrak X'$
(in fact this last observation follows easily from Result~\ref{res:4c} or Proposition~\ref{prop:2.1c}).
Thus the map $\theta_*$ of the preceding theorem is compatible with the induction of cells.
Moreover, the map $\theta_*$ enables us to obtain some additional information about the induction process for the particular cells involved ($\mathfrak C(\lambda)$ and $\mathfrak C(\lambda_*)$) by relating their rims.
In fact, knowledge of $Z(\lambda)$ enables us to determine the rim of $\mathfrak C(\lambda_*)$ and hence obtain reduced forms for all the elements of $\mathfrak C(\lambda_*)$.
The disadvantage of the above process is that the map $\theta_*$ is difficult to construct.

(ii) In view of Result~\ref{res:4c}, a consequence of the above is that the cell module corresponding to $\mathfrak C(\lambda_*)$ occurs as a constituent of the cell module corresponding to $\mathfrak C(\lambda)$ induced up to $S_{n+1}$.
This also agrees with the Branching Theorem (see~\cite[Theorem~9.2(i)]{Jam78}); this last observation follows once we combine the above with~\cite[Theorem~3.5]{MPa05} or~\cite[Theorem~4.6]{MPa15}.
\end{remark}

Finally for this subsection, we consider a dual construction to that of the composition
$\lambda_*$.
The \emph{upper star} composition $\lambda^*$ is the composition of
$n+1$ formed by prepending a new part 1 to the composition $\lambda$.
Thus, $\lambda^*=\dot\mu$  where $\mu=(\dot\lambda)_*$.
It is easy to see that, combining Theorem~\ref{TheoremReplaceThm6.1} with
Proposition~\ref{thm:3.13c}, we get the following corollary.
\begin{corollary}
\label{thm:6.3c}
Let $\lambda$ be a composition of $n$.
Then there is a natural injective mapping
$\theta^*\colon Z(\lambda)\rightarrow Z(\lambda^*)$
such that $z\theta^*\in (w_{S'}w_Szw_S w_{S'})w_{S'}\mathfrak{X}'w_{S'}$ for every $z\in Z(\lambda)$.
Moreover, $Y(\lambda)\theta^*\subseteq Y(\lambda^*)$. 
\end{corollary}
If $\mu$ is the composition defined before the statement of the
theorem, then $\theta^*$ is obtained
by composing the bijection of Proposition~\ref{thm:3.13c},
the injection of Theorem~\ref{TheoremReplaceThm6.1} for $\dot\lambda$ and the
bijection of Proposition~\ref{thm:3.13c} for $\mu$.

\subsection{A process relating to the restriction of cells}

We now consider a different type of `extension' of a composition.

{\bf Definition.}
Let $\lambda=(\lambda_1,\ldots,\lambda_r)$ be an $r$-part composition of $n$.

(i) Define $\max\lambda=\max\{\lambda_i\colon1\le i\le r\}$ and
$M(\lambda)=\{j\colon 1\le j\le r\mbox{ and }\lambda_j=\max\lambda\}$.
For $k\in M(\lambda)$, define the $r$-part composition $\lambda^{(k)}$
of $n+1$ by $\lambda^{(k)}_k=\lambda_k+1$ and
$\lambda^{(k)}_i=\lambda_i$ if $i\neq k$.

(ii) For $D$  a (principal) diagram with $\lambda_D=\lambda$ and $k\in M(\lambda)$, define a new diagram $D^{(k)}$ by $D^{(k)}=D\cup \{(k, c_D+1)\}$.
Then $D^{(k)}$ is a principal diagram with $\lambda_{D^{(k)}}=\lambda^{(k)}$.
Informally, to construct $D^{(k)}$ a new column is inserted in $D$ immediately after its last column and a single node is inserted in this new column on the $k$-th row.

{\bf Example}. If $\lambda$ $=$ $(1,2,2,1)$, then $Y(\lambda)$ $=$ $\{y_1,y_2,y_3\}$,
where
$y_1$ $=$ $[ 1$, $2$, $5$, $3$, $6$, $4 ]$,
$y_2$ $=$ $[ 3$, $1$, $4$, $2$, $5$, $6 ]$,
and
$y_3$ $=$ $[ 2$, $1$, $3$, $4$, $6$, $5 ]$.
Also, $Y_s(\lambda)$ $=$ $\{y_1,y_2\}$.
We display $\mathcal{E}^{(\lambda)}$ in Table~\ref{tbl:5c}.
\begin{table}[h]
\begin{center}
$\begin{array}{ccccc}
\begin{array}{cc}
   \times  \\
   \times & \times   \\
   \times & \times  \\
   \times  \\
\end{array}
&
\hspace{10pt}
&
\begin{array}{cc}
          & \times  \\
   \times & \times  \\
   \times & \times  \\
          & \times  \\
\end{array}
&
\hspace{10pt}
&
\begin{array}{ccc}
          & \times  \\
   \times & \times  \\
          & \times & \times  \\
          & \times  \\
\end{array}
\end{array}$
\end{center}
\caption{$\mathcal{E}^{(1,2,2,1)}$.}
\label{tbl:5c}
\end{table}
\par
In this case, $M(\lambda)=\{2,3\}$.
$\lambda^{(2)}=(1,3,2,1)$ and $\lambda^{(3)}=(1,2,3,1)$.
Moreover,
$Y(\lambda^{(2)})
$ $=$ $\{y^{(2)}_1,y^{(2)}_2,y^{(2)}_3,$
$y^{(2)}_4,y^{(2)}_5\}$
and
$Y(\lambda^{(3)})
$ $=$ $\{y^{(3)}_1,y^{(3)}_2,y^{(3)}_3,y^{(3)}_4,y^{(3)}_5\}$,
where
$y^{(2)}_1\!$ $=$ $\![1$, $2$, $5$, $7$, $3$, $6$, $4]$,
$y^{(2)}_2\!$ $=$ $\![3$, $1$, $4$, $7$, $2$, $5$, $6]$,
$y^{(2)}_3\!$ $=$ $\![2$, $1$, $3$, $7$, $4$, $6$, $5]$,
$y^{(2)}_4\!$ $=$ $\![4$, $1$, $3$, $5$, $2$, $6$, $7]$,
and
$y^{(2)}_5\!$ $=$ $\![3$, $1$, $2$, $4$, $5$, $7$, $6]$,
and
$y^{(3)}_1\!$ $=$ $\![4$, $2$, $5$, $1$, $3$, $6$, $7]$,
$y^{(3)}_2\!$ $=$ $\![2$, $3$, $6$, $1$, $4$, $7$, $5]$,
$y^{(3)}_3\!$ $=$ $\![3$, $2$, $4$, $1$, $5$, $7$, $6]$,
$y^{(3)}_4\!$ $=$ $\![1$, $2$, $6$, $3$, $5$, $7$, $4]$,
$y^{(3)}_5\!$ $=$ $\![2$, $1$, $3$, $4$, $6$, $7$, $5]$,
Also,
$Y_s(\lambda^{(2)})$ $=$ $\{y^{(2)}_1,y^{(2)}_2,y^{(2)}_4\}$
and
$Y_s(\lambda^{(3)})$ $=$ $\{y^{(3)}_1,y^{(3)}_2,y^{(3)}_4\}$.
We display $\mathcal{E}^{(\lambda^{(2)})}$ and
$\mathcal{E}^{(\lambda^{(3)})}$ in Table~\ref{tbl:6c} and
Table~\ref{tbl:7c}, respectively.
\begin{table}[h]
\begin{center}
$\begin{array}{ccccccccc}
\begin{array}{ccc}
   \times  \\
   \times & \times & \times  \\
   \times & \times  \\
   \times  \\
\end{array}
&
\hspace{10pt}
&
\begin{array}{ccc}
          & \times  \\
   \times & \times & \times  \\
   \times & \times  \\
          & \times  \\
\end{array}
&
\hspace{10pt}
&
\begin{array}{cccc}
          & \times  \\
   \times & \times &        & \times  \\
          & \times & \times  \\
          & \times  \\
\end{array}
&
\hspace{10pt}
&
\begin{array}{ccc}
          &        & \times  \\
   \times & \times & \times  \\
   \times &        & \times  \\
          &        & \times  \\
\end{array}
&
\hspace{10pt}
&
\begin{array}{cccc}
          &        & \times  \\
   \times & \times & \times  \\
          &        & \times & \times  \\
          &        & \times  \\
\end{array}
\end{array}$
\end{center}
\caption{$\mathcal{E}^{(1,3,2,1)}$.}
\label{tbl:6c}
\end{table}
\begin{table}[h]
\begin{center}
$\begin{array}{ccccccccc}
\begin{array}{ccc}
          &        & \times  \\
          & \times & \times  \\
   \times & \times & \times  \\
          &        & \times  \\
\end{array}
&
\hspace{10pt}
&
\begin{array}{ccc}
          & \times  \\
          & \times & \times  \\
   \times & \times & \times  \\
          & \times  \\
\end{array}
&
\hspace{10pt}
&
\begin{array}{cccc}
          &        & \times  \\
          & \times & \times  \\
   \times &        & \times & \times  \\
          &        & \times  \\
\end{array}
&
\hspace{10pt}
&
\begin{array}{ccc}
   \times  \\
   \times &        & \times  \\
   \times & \times & \times  \\
   \times  \\
\end{array}
&
\hspace{10pt}
&
\begin{array}{cccc}
          & \times  \\
   \times & \times  \\
          & \times & \times & \times  \\
          & \times  \\
\end{array}
\end{array}$
\end{center}
\caption{$\mathcal{E}^{(1,2,3,1)}$.}
\label{tbl:7c}
\end{table}

Observe that $\{D^{(2)}: D\in \mathcal{E}^{(1,2,2,1)}\}\subseteq \mathcal{E}^{(1,3,2,1)}$.
Note also that there are various injections $Y(\lambda)\rightarrow Y(\lambda^{(2)})$
given by
$y_1\mapsto y^{(2)}_1$,
$y_2\mapsto y^{(2)}_2$ or $y^{(2)}_4$,
and
$y_3\mapsto y^{(2)}_3$ or $y^{(2)}_5$.
These injections induce injections
$Y_s(\lambda)\rightarrow Y_s(\lambda^{(2)})$.
There is a similar situation for $\lambda^{(3)}$.
%

While the aim is to show that the situation in
the preceding example
generalizes, the following lemmas establish a complementary result.
%
\begin{lemma}\label{lem4.5n}
Let $\lambda\vDash n$ and $k\in M(\lambda)$.
Let $B,C\in \mathcal D^{(\lambda)}$ and let $B$ be an admissible diagram.
Let $\overline{B}=B^{(k)}$ and $\overline{C}=C^{(k)}$.
Then,\\
(i) $\overline{B}$ is an admissible diagram.\\
(ii) If $w_C=w_B$ then $w_{\overline{B}}=w_{\overline{C}}$.
\end{lemma}

\begin{proof}
(i) Let $\nu$ be the subsequence type of type $\overline{B}$ and let $\overline{\lambda}=\lambda^{(k)}$.
Since $\lambda_{\overline{B}}=\overline{\lambda}$ we have $\nu\unlhd\overline{\lambda}$ by Proposition~\ref{prop:3.8c}.
Write $\lambda'=(\lambda_1',\ldots,\lambda_{r'}')$.
Since $B$ is admissible, we see from the construction of $\overline{B}$ that $(\lambda_1',\ldots,\lambda_{r'}',1)\unlhd \nu$.
The assumption $k\in M(\lambda)$ ensures that $\overline{\lambda}'=(\lambda_1',\ldots,\lambda_{r'}',1)$, hence $\overline{\lambda}'\unlhd \nu$.
We conclude that $\nu=\overline{\lambda}'$ and $\overline{B}$ is admissible.

(ii) is immediate.
\end{proof}

\begin{lemma}
\label{lem:6.4c}
Let $\lambda\vDash n$,
$z\in Z(\lambda)$, $D=D(z,\lambda)$ and
$k\in M(\lambda)$.
Also let $\overline{\lambda}=\lambda^{(k)}$,
$\overline{D}=D^{(k)}$ and $\overline{z}=w_{\overline{D}}$.
We then have

(i) $\overline{D}$ is an admissible diagram and $\overline{z}\in Z(\overline{\lambda})$.
(Hence, there is a well-defined map $\theta$ from $Z(\lambda)$ to $Z(\overline{\lambda})$ in which $z\mapsto\overline{z}$.)

(ii)
If $z\not\in Y(\lambda)$ then
$\overline{z}\not\in Y(\overline{\lambda})$.
\end{lemma}
\begin{proof}
Assume the hypothesis.

(i) We have $z=w_D$ and $z\in Z(\lambda)$ so $D$ is admissible by Theorem~\ref{thm:3.9c}.
Invoking Lemma~\ref{lem4.5n}(i) we see that $\overline{D}$ is also admissible and hence $\overline{z}=w_{\overline{D}}\in Z(\overline{\lambda})$.


(ii)
Suppose now that $z\not\in Y(\lambda)$.
Since $z\in Z(\lambda)$, $z=w_D$, $\lambda_D=\lambda$ and
$D$ is admissible,
we get from Proposition~\ref{prop:3.12c} that there is an
admissible diagram $E\in\mathcal{D}^{(\lambda)}$ such that
$w_E\neq w_D$ and $t^Ew_D$ is a standard $E$-tableau.
The fact that $w_E\ne w_D$ ensures that $t=t^Ew_D\ne t_E$.
Let $F=E^{(k)}$.
Then $\lambda_{F}=\overline{\lambda}$ and $F$ is admissible by Lemma~\ref{lem4.5n}(i).
Now let $t'=t^{F}w_{\overline{D}}$.
Then $(u)t'=(u)t$ for all $u\in E$ and $(k,c_E+1)t'=n+1$ so $t'$ is a standard $F$-tableau.
It is also easy to see from this construction that $(u)t_F=(u)t_E$ for all $u\in E$.
The fact that $t\ne t_E$ ensures that $t^Fw_{\overline{D}}=t'\ne t_F=t^Fw_F$.
We conclude that $w_{\overline{D}}\ne w_F$ and this is enough to complete the proof in view of Proposition~\ref{prop:3.12c}.
\end{proof}

As motivation for our next result, we now consider the case $\lambda=(2,1,1,2)$, in which
$1\in M(\lambda)$.
Here, $Y(\lambda)$ $=$ $\{y_1,y_2,y_3\}$, where
$y_1$ $=$ $[ 1$, $5$, $2$, $3$, $4$, $6 ]$,
$y_2$ $=$ $[ 1$, $3$, $4$, $5$, $2$, $6 ]$,
and
$y_3$ $=$ $[ 1$, $4$, $2$, $5$, $3$, $6 ]$,
Also, $Y_s(\lambda)$ $=$ $\{y_1,y_2\}$.
We display $\mathcal{E}^{(\lambda)}$ in Table~\ref{tbl:8c}.
\begin{table}[h]
\begin{center}
$\begin{array}{ccccc}
\begin{array}{cc}
  \times & \times \\
  \times \\
  \times \\
  \times & \times \\
\end{array}
&
\hspace{10pt}
&
\begin{array}{cc}
  \times & \times \\
         & \times \\
         & \times \\
  \times & \times \\
\end{array}
&
\hspace{10pt}
&
\begin{array}{ccc}
  \times & \times \\
  \times \\
         & \times \\
  \times & \times \\
\end{array}
\end{array}$
\end{center}
\caption{$\mathcal{E}^{(2,1,1,2)}$.}
\label{tbl:8c}
\end{table}
\par
In this case,
$\lambda^{(1)}=(3,1,1,2)$.
Moreover,
$Y(\lambda^{(1)})
$ $=$ $\{y^{(1)}_1,y^{(1)}_2,y^{(1)}_3\}$
where
$y^{(1)}_1\!$ $=$ $\![ 1$, $5$, $7$, $2$, $3$, $4$, $6 ]$,
$y^{(1)}_2\!$ $=$ $\![ 1$, $3$, $7$, $4$, $5$, $2$, $6 ]$,
and
$y^{(1)}_3\!$ $=$ $\![ 1$, $4$, $7$, $2$, $5$, $3$, $6 ]$,
and $Y_s(\lambda^{(1)})=\{y^{(1)}_1,y^{(1)}_2\}$.
We find in this case that there is a bijection
$Y(\lambda)\rightarrow Y(\lambda^{(1)})$,
which induces a bijection
$Y_s(\lambda)\rightarrow Y_s(\lambda^{(1)})$.
We display $\mathcal{E}^{(\lambda^{(1)})}$ in Table~\ref{tbl:9c}.
\begin{table}[h]
\begin{center}
$\begin{array}{ccccc}
\begin{array}{ccc}
  \times & \times & \times \\
  \times \\
  \times \\
  \times & \times \\
\end{array}
&
\hspace{10pt}
&
\begin{array}{ccc}
  \times & \times & \times \\
         & \times \\
         & \times \\
  \times & \times \\
\end{array}
&
\hspace{10pt}
&
\begin{array}{cccc}
  \times & \times & \times \\
  \times  \\
         & \times \\
  \times & \times \\
\end{array}
\end{array}$
\end{center}
\caption{$\mathcal{E}^{(3,1,1,2)}$.}
\label{tbl:9c}
\end{table}
\par
Since $4\in M(\lambda)$, we get a similar analysis and result.
Here, $\lambda^{(4)}=(2,1,1,3)$ and
$Y(\lambda^{(4)}) $ $=$ $\{y^{(4)}_1,y^{(4)}_2,y^{(4)}_3\}$,
where
$y^{(4)}_1\!$ $=$ $\![ 2$, $4$, $5$, $6$, $1$, $3$, $7 ]$,
$y^{(4)}_2\!$ $=$ $\![ 2$, $5$, $3$, $6$, $1$, $4$, $7 ]$,
and
$y^{(4)}_3\!$ $=$ $\![ 2$, $6$, $3$, $4$, $1$, $5$, $7 ]$,
and
$Y_s(\lambda^{(4)}) $ $=$ $\{y^{(4)}_1,y^{(4)}_3\}$.
There is, however, a very close connection between $Y(\lambda^{(1)})$
and  $Y(\lambda^{(4)})$, since every diagram
$D\in\mathcal{D}^{(\lambda^{(1)})}$ corresponds, on rotation through
$180^{\circ}$, to a diagram $\dot D$ with
$\dot D\in\mathcal{D}^{(\lambda^{(4)})}$, and this correspondence is
bijective.
We display $\mathcal{E}^{(\lambda^{(4)})}$ in Table~\ref{tbl:10c}.
\begin{table}[h]
\begin{center}
$\begin{array}{ccccc}
\begin{array}{ccc}
          & \times & \times \\
          &        & \times \\
          &        & \times \\
   \times & \times & \times \\
\end{array}
&
\hspace{10pt}
&
\begin{array}{ccc}
          & \times & \times \\
          & \times \\
          &        & \times \\
   \times & \times & \times \\
\end{array}
&
\hspace{10pt}
&
\begin{array}{cccc}
          & \times & \times \\
          & \times \\
          & \times \\
   \times & \times & \times \\
\end{array}
\end{array}$
\end{center}
\caption{$\mathcal{E}^{(2,1,1,3)}$.}
\label{tbl:10c}
\end{table}
\par
\begin{theorem}
\label{thm:6.5c}
Let $\lambda\vDash n$, let $k\in M(\lambda)$ and
let $\overline{\lambda}=\lambda^{(k)}$.
Also let $d=s_{p+1}\dots s_n$ and $\overline{d}=s_{q+1}\dots s_n$, where $p=\sum_{1\le i\le k}\lambda_i$ and $q=\sum_{1\le i\le k-1}\lambda_i$ (so that $q=0$ when $k=1$).
Finally, let $\theta: Z(\lambda)\to Z(\overline{\lambda}):$ $z\mapsto \overline{z}$ be the map defined in Lemma~\ref{lem:6.4c}.
Then,

(i) For each $z\in Z(\lambda)$ we have $z\theta\,(=\overline{z})=dz$ (hence $\theta$ is injective),

(ii) $\overline{d}\mathfrak{C}(\lambda) \subseteq\mathfrak{C}(\overline{\lambda})$,

(iii) Any one of conditions (a), (b) or (c) below is sufficient to ensure that $Y(\lambda)\theta\subseteq Y(\overline{\lambda})$.

(a) $k=1$,

(b) $k\neq\max M(\lambda)$,

(c) $M(\lambda)=\{k\}$ and $\max\{i\colon\lambda_i=m\}>k$
where $m=\max\{\lambda_j\colon 1\le j\le r\mbox{ and }j\neq k\}$.

Moreover, in case $k=1$, we have equality $Y(\lambda)\theta=Y(\overline{\lambda})$ and $\theta$ also induces a bijective mapping
$\mathcal{E}^{(\lambda)}\rightarrow\mathcal{E}^{(\overline{\lambda})}$ in which $D\mapsto D^{(k)}$.

\end{theorem}

\begin{proof}
Let $\lambda$, $k$, $\overline{\lambda}$, $d$, $\overline{d}$ and $\theta$ be as in the statement of the theorem.
For each $z\in Z(\lambda)$, we define $\overline{z}=w_{\overline{D}}$
where $D=D(z,\lambda)$ and $\overline{D}=D^{(k)}$.
From the construction of $\overline{D}$ we have
$i\overline{z}=iz$ if $1\le i\le p$,
$(p+1)\overline{z}=n+1$,
and $i\overline{z}=(i-1)z$ if $p+2\le i\le n+1$.
Thus, $z\theta=\overline{z}=dz$, where $d=(n+1,n,\dots,p+1)=s_{p+1}\dots s_n$.
It follows that if $z_1,z_2\in Z(\lambda)$ and $z_1\neq z_2$
then $\overline{z_1}\neq\overline{z_2}$.
Hence the mapping $z\mapsto\overline{z}$ is injective.
By a straightforward calculation, it can be shown that
$\overline{d}w_{J(\lambda)}=w_{J(\overline{\lambda})}d$.
Hence, $\overline{d}\mathfrak{C}(\lambda)$
$=\overline{d}w_{J(\lambda)}Z(\lambda)$
$=w_{J(\overline{\lambda})}dZ(\lambda)$
$\subseteq w_{J(\overline{\lambda})}Z(\overline{\lambda})$
$=\mathfrak{C}(\overline{\lambda})$.
This completes the proof of items (i) and (ii) of the theorem.

In order to prove item (iii) suppose, in addition, that any one of conditions (a), (b) or (c) holds.
Also let $y\in Y(\lambda)$, and set $\overline{y}=w_{\overline{D}}$ where $D=D(y,\lambda)$ and $\overline{D}=D^{(k)}$.
We suppose that $\overline{y}\not\in Y(\overline{\lambda})$ and aim to obtain a contradiction.
By Proposition~\ref{prop:3.12c}, there exists a diagram $F$ of subsequence type
$\overline{\lambda}'=(\lambda_1',\ldots,\lambda_{r'}',1)$, where
$\lambda'=(\lambda_1',\ldots,\lambda_{r'}')$, satisfying
$\lambda_F=\overline{\lambda}$ and $w_F\neq w_{\overline{D}}$ and a
standard $F$-tableau $t=t^Fw_{\overline{D}}\ne t^Fw_F=t_F$.
We may suppose that $F$ has no empty columns, that is $F$ is principal.
The last entry of the $k$-th row of $t$ is $(p+1)\overline{y}=n+1$.
Let $(k,j)$ be the node of $F$ for which $(k,j)t=n+1$.
Since $t$ is standard, the only node $(i',j')$ of $F$ with $i'\ge k$
and $j'\ge j$ is the node $(k,j)$.
Hence, any path in $F$ containing the node $(k,j)$ terminates in this
node.
We construct a diagram $E$ from $F$ by removing the node $(k,j)$ and
use $E$ to obtain a contradiction.
The tableau $\hat t$ obtained from $t$ by removing the entry $(k,j)t$ is a
standard $E$-tableau and $\hat t=t^Ew_D$.
Clearly, $\lambda_E=\lambda$.
Let $\nu$ be the subsequence type of $E$.
By Proposition~\ref{prop:3.8c}, $\nu\unlhd\lambda'$.	
We deal separately with the cases (a)--(c).

(a) Here, $k=1\in M(\lambda)$.
Then $n+1$ is the only entry in the last column of $t$.
Hence, $(1,c_F)$ is a node of $F$ and is the only node in the last
column of $F$.
Since the only path in $F$ containing $(1,c_F)$ has length 1,
$\nu=\lambda'$ and $E$ is admissible.

(b) 
Now let $l=\max M(\lambda)$ so that $k<l$ and $\lambda_k=\lambda_l=r'$.
Let $\Pi_u$ be a $u$-path in $F$ of length
$\sum_{1\le j\le u}\lambda_j'$ where $1\le u\le r'$.
Then, by Proposition~\ref{prop:3.7c}, $\Pi_u$ has $u$ nodes on each
of the $k$-th and $l$-th rows.
Thus, all $u$ paths of $\Pi_u$ contain nodes on the $l$-th row.
Since a path containing node $(k,j)$ must terminate at it, $(k,j)$
cannot be a node of $\Pi_u$.
Hence, $\Pi_u$ is a $u$-path of $E$.
This proves that $E$ is admissible.
By considering $\Pi_{r'}$, we see that all nodes of $E$ on the $l'$-th
row, where $l'<l$, are in columns containing nodes of $E$ on the
$l$-th row or to the left of such columns.
Hence, node $(k,j)$ is the only node of $F$ on the $j$-th column and
there are no nodes of $E$ on or to the right of the $j$-th column.

(c)
In this case, we first let $u=\lambda_l$ where
$l=\max\{i\colon\lambda_i=m\}$ and let $\Pi$ be a $u$-path of $F$ of
length $\sum_{1\le j\le u}\lambda_j'$.
By Proposition~\ref{prop:3.7c}, $\Pi$ contains all nodes on the
$l$-th row of $F$ and by the argument in (b), $\Pi$ cannot contain
node $(k,j)$.
Hence, $\Pi$ is a $u$-path in $E$.

Since $\lambda_j'=1$ for $u<j\le r'$, we can form a $u'$-path in $E$
of length $\sum_{1\le j\le u'}\lambda_j'$ for any $u'$ such that
$u<u'\le r'$ by adding $u'-u$ paths of length 1 whose nodes are
chosen arbitrarily from the $r'-u$ nodes on the $k$-th row of $E$ which
are not in $\Pi$.

Finally, if $1\le u'<u$, let $\Pi$ be a $u'$-path of $F$ of
length $\sum_{1\le j\le u'}\lambda_j'$.
By Proposition~\ref{prop:3.7c}, $\Pi$ contains exactly $u'$ nodes on
the $l$-th row of $F$ and again by the argument in (b), $\Pi$ cannot
contain node $(k,j)$.
Hence, $\Pi$ is a $u'$-path in $E$.
So, in this case, $E$ is again admissible.
By an argument similar to that in case (b), we again find
that node $(k,j)$ is the only node of $F$ on the $j$-th column and
there are no nodes of $E$ on or to the right of the $j$-th column.

We complete cases (a)--(c) by noting that, since
$w_{\overline{D}}\neq w_F$ and $n+1$ appears as the sole entry in
the last columns of $t_{\overline{D}}$ and $t_F$, we get $w_D\neq w_E$.
By Proposition~\ref{prop:3.12c}, $y\not\in Y(\lambda)$ contrary to
hypothesis.


Next we establish that $Y(\lambda)\theta=Y(\overline{\lambda})$ in the case $k=1$.
Let $x\in Y(\overline{\lambda})$.
We must show that $x=y\theta$ for some $y\in Y(\lambda)$.
Let $C=D(x,\overline{\lambda})$.
Then $\lambda_C=\overline{\lambda}$
and, since $x\in Z(\overline{\lambda})$, $C$ is of subsequence type
$\overline{\lambda}'$ by Theorem~\ref{thm:3.9c}(ii).
Since $\overline{\lambda}_1>\overline{\lambda}_i$ for $i>1$,
$\overline{\lambda}_1$ is the number of parts of the partition
$\overline{\lambda}'$ and the last part of $\overline{\lambda}'$ is $1$.
Let $B$ be the diagram obtained from $C$ by removing the last node of
the first row, and let $\nu$ be its subsequence type.
Then $\lambda_B=\lambda$.
So, $\nu\unlhd\lambda'$ by Proposition~\ref{prop:3.8c}.

Let $\Pi$ be a $u$-path in $C$ where $1\le u<\overline{\lambda}_1$,
and let $(1,j_1),\dots,(1,j_u)$ be the first $u$ nodes on the first
row of $C$.
By Proposition~\ref{prop:3.7c}, $\Pi$ has exactly $u$ nodes on the first row of
$C$; call these nodes $(1,j_1')$, \dots , $(1,j_u')$.
Then $\Pi'=(\Pi\backslash \{(1,j_1')$, \dots , $(1,j_u')\})\cup
\{(1,j_1)$, \dots , $(1,j_u)\}$ forms a $u$-path in $C$ of the
same length as $\Pi$.
Since $\Pi'$ does not contain the final node of the first row of $C$,
$\Pi'$ is a $u$-path of $B$.
Since we can choose $\Pi$ to have length
$\sum_{1\le j\le u}\bar\lambda_j'=\sum_{1\le j\le u}\lambda_j'$,
we find a $u$-path of this length in $B$.
Hence, $B$ is admissible,
$\nu=\lambda'$ and $w_B\in Z(\lambda)$ by
Theorem~\ref{thm:3.9c}(ii).

Let $\overline{B}=B^{(1)}$. Then $\overline{B}$ is admissible by Lemma~\ref{lem4.5n}(i), since $B$ is admissible.
We want to show that $C=\overline{B}$.
Suppose on the contrary that $C\ne\overline{B}$.
In view of the way $C=D(x,\overline{\lambda})$ is constructed from $x$ and $\overline{\lambda}$ (and also the way $B$ is constructed from $C$) we would then have that the last node on the first row of $C$ is not in the last column of $C$ or is not the sole node in the last column of $C$.
It follows from this that $w_C\ne w_{\overline{B}}$.
On the other hand $t^{\overline{B}}w_C$ is a standard $\overline{B}$-tableau since $t^{\overline{B}}w_C$ is constructed from
$t_C=t^Cw_C$ by moving the last entry on the first row of $t_C$ to the
position of the last node of the first row of $\overline{B}$ and
keeping all other entries of $t_C$ fixed.
This contradicts the fact that $x=w_C\in Y(\overline{\lambda})$.
Hence $C=\overline{B}=B^{(1)}$.

%
Let $y=w_B$. Then $w_{\overline{B}}=y\theta$ from Lemma~\ref{lem4.5n}(ii) (or by noting that $B=D(w_B,\lambda)$).
As $w_{\overline{B}}=w_C=x$, we get $x=y\theta$.
Since $B$ is admissible, $y\in Z(\lambda)$.
Since $y\theta=x\in Y(\overline{\lambda})$, it now follows from
Lemma~\ref{lem:6.4c}(ii) that $y\in Y(\lambda)$.
This completes the proof that in the case $k=1$ we have equality $ Y(\lambda)\theta=Y(\overline{\lambda})$.
Finally, it is clear that $\theta$ induces a bijective mapping $\mathcal{E}^{(\lambda)}\rightarrow\mathcal{E}^{(\overline{\lambda})}$, given by $D\mapsto D^{(1)}$.
%
%
\end{proof}

We illustrate Theorem~\ref{thm:6.5c} with the following example.
Consider $\lambda=(2,1,2,2)$.
Then $M(\lambda)=\{1,3,4\}$ and
$Y(\lambda)=\{(2,4)(3,5,6),(2,5,6,4,3)\}$.
If $k=1$, $\overline{\lambda}=(3,1,2,2)$, $Y(\overline{\lambda})$
$=\{(2,4,5)(3,8,7),(2,5,3,8,7,4)\}$, and $d=(8,7,6,5,4,3)$ is a prefix
of both elements of $Y(\overline{\lambda})$ and
$d^{-1}Y(\overline{\lambda})=Y(\lambda)$.
If $k=3$, $\overline{\lambda}=(2,1,3,2)$, and $Y(\overline{\lambda})$
$=\{(1,2,5,3,6,7,4),$ $(2,5,4)(3,6,7),$ $(1,2,6,7,5,4),$
$(2,4)(3,5,6,8,7),$ $(2,5,6,8,7,4,3),$ $(2,6,7,4,3)\}$.
The element $d=(8,7,6)$ is a prefix of the fourth and fifth elements
of $Y(\overline{\lambda})$ and $dY(\lambda)$ consists of these two
elements.
If $k=4$, $\overline{\lambda}=(2,1,2,3)$, and $Y(\overline{\lambda})$
$=\{(1,2,5,7,4,3,6),(1,2,6)(5,7)\}$.
Neither of the elements of $Y(\lambda)$ is a suffix of either of these
elements.

Next we comment on the link Theorem~\ref{thm:6.5c} provides with the restriction of cells.

\begin{remark}
(i) Comparing Theorem~\ref{thm:6.5c} (and keeping the notation introduced there) with the `right analogue' of Result~\ref{res:3c} we see, under the assumption $k\in M(\lambda)$, that the Kazhdan-Lusztig cell $\mathfrak C(\lambda)$ is one of the cells appearing in the decomposition of $\mathfrak C(\overline{\lambda})$ into a disjoint union of sets of the form $e\mathcal{D}$ where $\mathcal{D}$ is a right cell of $S_n$ and $e\in\mathfrak{X}'{}^{-1}$.
To see this, note first that $\mathfrak{X}'{}^{-1}$ is the set of distinguished left coset representatives of $S_n$ into $S_{n+1}$.
Now $\overline{d}\in\mathfrak{X}'{}^{-1}$, so the above is an immediate consequence of the uniqueness of decomposition of any $x'\in S_{n+1}$ in the form $e'x$ where $e'\in\mathfrak{X}'{}^{-1}$ and $x\in S_n$.
\\
The theorem also provides information about which particular element of $\mathfrak{X}'{}^{-1}$ we need to premultiply the elements of $\mathfrak C({\lambda})$ for the particular cases of cells we are considering.
Note that this information is provided by Proposition~\ref{prop:2.2c} in the more general case, however this involves applying the reverse Robinson-Schensted process.
See also~\cite[Theorem~3.3]{Jacon2006} for a different result related to the above discussion.

(ii) In addition to providing a link with the restriction of cells, the theorem also relates the rims of the cells involved.
If we know the rim of $\mathfrak{C}(\lambda)$ we can immediately deduce information about the rim of $\mathfrak{C}(\overline{\lambda})$.
In particular, when $k=1$, mere knowledge of the rim of $\mathfrak{C}(\lambda)$ allows us to determine completely the rim of $\mathfrak{C}(\overline{\lambda})$ and, as a consequence, obtain reduced forms for all the elements in this cell (as the remaining elements of $\mathfrak{C}(\overline{\lambda})$ are the products of $w_{J(\overline{\lambda})}$ with the prefixes of the elements in its rim).

(iii) In view of Result~\ref{res:3c}, another consequence of the above (when $k\in M(\lambda)$) is that the cell module corresponding to $\mathfrak{C}(\lambda)$ is a constituent of the cell module corresponding to $\mathfrak{C}(\overline{\lambda})$ when restricted to $S_n$.
This can also be seen to agree with the Branching Theorem (see~\cite[Theorem 9.2(ii)]{Jam78}) by using the same results as for the induction case earlier on.
\end{remark}

Finally, by applying Proposition~\ref{thm:3.13c}, we can obtain a dual result to Theorem~\ref{thm:6.5c} in a similar manner to how Corollary~\ref{thm:6.3c} is obtained from Theorem~\ref{TheoremReplaceThm6.1}.
In particular, if $\lambda$ is a composition of $n$ with $r$ parts and $r\in M(\lambda)$, there is a bijection between the elements of $Y(\lambda)$ and the elements of $Y(\lambda^{(r)})$.

\end{section}

%


\end{document}